\documentclass{amsart}

\usepackage{verbatim}
\usepackage{xcolor}
\usepackage{tikz}
\usepackage{tikz-cd}
\usepackage{arydshln}
\usepackage{caption}
\usepackage{subcaption}
\usepackage{graphicx}

\tikzset{
	symbol/.style={
		draw=none,
		every to/.append style={
			edge node={node [sloped, allow upside down, auto=false]{$#1$}}}
	}
}

\usepackage{amsmath}
\usepackage{bbm}
\usepackage{hyperref}
\usepackage{times,amsfonts,amsmath,amstext,amsbsy,amssymb,
amsopn,amsthm,upref,eucal}
\usepackage[T1]{fontenc}

\newtheorem{theorem}{Theorem}[section]
\newtheorem{proposition}[theorem]{Proposition}
\newtheorem{lemma}[theorem]{Lemma}
\newtheorem{corollary}[theorem]{Corollary}

\theoremstyle{definition}
\newtheorem{definition}[theorem]{Definition}
\newtheorem{example}[theorem]{Example}

\theoremstyle{remark}
\newtheorem{remark}[theorem]{Remark}

\newtheorem{question}[theorem]{Question}

\numberwithin{equation}{section}

\newcommand{\field}[1]{\mathbb{#1}}
\newcommand{\R}{\field{R}}
\newcommand{\N}{\field{N}}
\newcommand{\C}{\field{C}}
\newcommand{\Z}{\field{Z}}

\newcommand{\Cal}{\mathcal}
\newcommand{\A}{\Cal{A}}
\newcommand{\V}{\Cal{V}}

\begin{document}

\title[From limit theorems to mixing limit theorems]{From limit theorems  to mixing limit theorems}



\author{Francisco Arana--Herrera}

\email{farana@umd.edu}

\address{Department of Mathematics, William E. Kirwan Hall, 4176 Campus Dr, University of Maryland, College Park, MD 20742, USA}
	
\author{Giovanni Forni}

\email{gforni@umd.edu}

\address{Department of Mathematics, William E. Kirwan Hall, 4176 Campus Dr, University of Maryland, College Park, MD 20742, USA}



\date{}

\begin{abstract}
    Motivated by work of Dolgopyat and Nándori, we establish a general method for upgrading limit theorems for Birkhoff sums and cocycles over dynamical systems to mixing limit theorems under mild ergodicity and hyperbolicity assumptions.  Building on previous work of Al-Saqban and Forni, we apply this method to obtain mixing limit theorems for particular subbundles of the Kontsevich-Zorich cocycle. In forthcoming work of Arana-Herrera and Honaryar these results are applied to study the arithmetic/homological complexity of long simple closed geodesics on negatively curved surfaces.
\end{abstract}

\maketitle


\thispagestyle{empty}

\tableofcontents

\section{Introduction}

\subsection*{Motivation.} Several important conditions can be used to describe the randomness of a deterministic dynamical system. For instance, ergodicity ensures the almost everywhere convergence of time averages to spatial averages. Other limit theorems like central limit theorems can then be used to characterize the distribution of the deviations of time averages from spatial averages. Mixing, a notion strictly stronger than ergodicity, guarantees the asymptotic independence between the initial position of points of a system and the position of their forward iterates. Joint generalizations of these notions of randomness, called mixing limit theorems, were introduced by Dolgopyat and Nándori in \cite{DN20}. The main goal of this paper is to develop a method for upgrading limit theorems for Birkhoff sums and cocycles over dynamical systems to mixing limit theorems under mild ergodicity and hyperbolicity assumptions.

In the case of the central limit theorem, the question studied in this paper is of interest in itself. For instance, for Birkhoff sums: Does the distribution of the deviations of time averages from spatial averages remain unchanged if we prescribe the initial and final positions of points? In this paper we show the answer to this question is positive under weak ergodicity and hyperbolicity conditions.

The Oseledets ergodic theorem interpreted in the distributional sense is another example of a limit theorem. As an application of the methods developed in this paper, and building on previous work of Al-Saqban and Forni \cite{KZCLT}, we prove mixing Oseledets ergodic theorems and central limit theorems for particular subbundles of the Kontsevich-Zorich cocycle. By work of Bell, Delecroix, Gadre, Gutiérrez-Romo, and Schleimer \cite{flow_group}, these results hold for the invariant part of the Kontsevich-Zorich cocycle over loci of orientation double covers of quadratic differentials. 

In forthcoming work of Arana-Herrera and Honaryar \cite{statistics} these results are applied to study the arithmetic/homological complexity of long simple closed geodesics on negatively curved surfaces. More concretely, laws of large numbers and central limit theorems are proved. In the class of all closed geodesics Sharp has proved stronger local limit theorems \cite{sharp}; see also \cite{hom1,hom2,hom3,hom4,hom5} for previous related results. The situation in the case of simple closed geodesics is much less understood and, to the knowledge of the authors, the results discussed in \cite{statistics} are the first of its kind. We highlight that, as a consequence of these results, local limit theorems in the spirit of Sharp's work but for simple closed geodesics cannot hold without significant modifications.

The study of central limit theorems for Birkhoff sums and integrals over uniformly hyperbolic systems goes back to works of Sinai, Ruelle, Bowen, and Ratner \cite{Sinai1, Sinai2, Sinai3, Ruelle1, Bowen1, Bowen2, Bowen3, Ratner1}. For more recent works, including the study of the case of partially hyperbolic systems, see \cite{Cher, Liv, AaDen, Dol1, Gouezel}. For a study of central limit theorems for Birkhoff integrals over horocycle flows on hyperbolic surfaces see \cite{BufFor}. Local central limit theorems, a stronger notion, have been studied more sparingly \cite{Wad, Iwa, DN16, DN17, DN20}. To the knowledge of the authors, while there is a rich literature on non-Abelian central limit theorems for random matrix products, or, equivalently, for cocycles over Bernoulli systems \cite{Be54, FK60, LeP82, BL85, GR86, GM89, BQ16}, there are few results for deterministic cocycles: Furman and Kozma \cite{FK}, and Park and Piraino \cite{PP22}, proved CLTs for typical cocycles  over shifts of finite type in the sense of Bonatti and Viana \cite{PP22}. Al-Saqban and Forni \cite{KZCLT} considered the special case of the Kontsevich-Zorich cocycle over the Teichmüller flow. In the latter case, the base dynamics is non-uniformly hyperbolic and can be reduced to a suspension flow over a Markov chain with countably many states.

The main difficulty of the problem at hand comes from the fact that the notion of convergence in distribution is rather weak and does not seem to interact nicely with direct applications of the mixing property. Rather than following a direct approach, we reduce our problem to the study of the limit points of an appropriately defined sequence of random variables. The hyperbolicity and ergodicity assumptions guarantee such limit points are invariant under an ergodic transformation. In particular, such limit points are constant. To control the value of such constants we apply the assumed limit theorem.

Other than the applications discussed in this paper and in \cite{statistics}, the authors expect the notion of mixing limit theorems to play a relevant role in the study of randomness in dynamical systems. This can already be seen in the work of Dolgopyat and Nándori \cite{DN20} that inspired this paper. Further applications should arise in the future, especially as the theory of limit theorems for cocycles over dynamical systems is further developed.

\subsection*{Birkhoff sums.}

Let $(X,d)$ be a metric space supporting a Borel probability measure $\mu$ and an invertible, measure-preserving transformation $T\colon X \to X$. 

Given a measurable function $f\colon X \to \R$, the ergodic sums of $f$ with respect to $T$ are said to satisfy a (spatial) distributional limit theorem (DLT) on $(X, \mu)$ if there exists a sequences $\A:=(A_N)_{N \in \mathbb{N}}$ of real numbers and $\V:=(V_N)_{N \in \mathbb{N}}$ of positive real numbers with $V_N \to \infty$ as $N \to \infty$, and a random variable $S$, such that the random variables
\[
S_N(x):=\frac{\sum_{n=0}^{N-1} f(T^n x)- A_N}{V_N} \ \ \text{on } (X,\mu)
\]
converge in distribution to $S$ as $N \to \infty$, that is, for every interval $(a,b) \in \mathbb{R}$ with $\mathbb{P}(S \in \{a,b\}) = 0$, the following holds,
\[
\lim_{N \to \infty} \mu(\{x \in X \ | \ S_N(x) \in (a,b) \}) = \mathbb{P}(S \in (a,b)).
\]
We usually refer to $\A$ as the averaging sequence, $\V$ as the normalizing sequence, and $S$ as the limiting distribution.

\begin{example}
Birkhoff's ergodic theorem guarantees that, if $\mu$ is ergodic with respect to $T$, then the ergodic sums of any measurable, bounded function $f \colon X \to \mathbb{R}$ satisfy a DLT with averaging sequence $\mathcal{A}$ given by $A_N :=  N \cdot \mu(f)$ for every $N \in \mathbb{N}$,  normalizing sequence $\V$ given by $V_N := N$ for every $N \in \mathbb{N}$, and $S$ a (constant) random variable with distribution the Dirac mass at 
$0 \in \R$. 
\end{example}

\begin{example}
Another important case in hyperbolic dynamics is the central limit theorem (CLT)
for which the averaging sequence $\mathcal{A}$ is given by $A_N :=  N \cdot \mu(f)$ for every $N \in \mathbb{N}$, the normalizing sequence $\V$ is given by $V_N = \sqrt{N}$ for every $N \in \mathbb{N}$, and the random variable $S$ has normal distribution with mean $0$. However, it is known that there exist mixing dynamical systems of zero entropy 
for which the CLT does not hold \cite{FF03}, as well as systems with zero entropy for which it does hold
\cite{DDKN,DFK}. 
In general, limit theorems for systems of zero
entropy are not well-understood.
\end{example}

Now let $U \colon X \to X$ be a measure-preserving, ergodic transformation . We say the pair $(T,U)$ is contracting if for $\mu$-almost-every $x \in X$,
\[
\lim_{n \to \infty} d(T^n U x, T^n x) = 0.
\]
We say that a function $f \colon X \to \R$ is $(T,U, \V)$-adapted if for $\mu$-almost-every $x \in X$ and every $N\in \N$ we have
\[
\sum_{n=0}^N |f(T^n U x) - f(T^n x)| = o_x(V_N).
\]

\begin{remark}
The pair $(T,U)$ is called strongly contracting if  for $\mu$-almost-every $x \in X$
$$
\sum_{n=0}^\infty d(T^n U x, T^nx) < +\infty\,
$$
and hyperbolic if for $\mu$-almost every $x \in X$ there exist  $C >0$ and $\kappa > 0$ such that 
\[
 d(T^n U x, T^nx) \leq C e^{-\kappa n} \ \ \text{for every } n \in \mathbb{N}.
\]
If $(T,U)$ is strongly contracting, then any Lipschitz function $f \colon X \to \R$ is $(T,U, \V)$-adapted for any diverging sequence $\V := (V_N)_{N \in \mathbb{N}}$. Furthermore, for $\mu$-almost every $x \in X$ and every $N \in \mathbb{N}$, 
\[
\sum_{n=0}^N |f(T^n U x) - f(T^n x)| = O_x(1).
\]
\end{remark}

The following is the main result of this paper for Birkhoff sums.

\begin{theorem}
\label{theo:metric_MCLT_intro}
Let $(X,d)$ be a metric space supporting a Borel probability measure $\mu$, let $T\colon X \to X$ be an invertible, measure-preserving transformation, and let $f \colon X \to \mathbb{R}$ be a bounded measurable function. Assume that the Birkhoff sums of $f$ with respect to $T$ satisfy a DLT on $(X,\mu)$ with averaging sequence $\A :=(A_N)_{N \in \mathbb{N}}$, normalizing sequence $\V :=(V_N)_{N \in \mathbb{N}}$, and limiting distribution $S$ in the sense that the random variables
\[
S_N(x) := \frac{\sum_{n=0}^{N-1}f(T^n x) - A_N}{V_N} \ \ \text{on } (X,\mu)
\]
converge in distribution to $S$ as $N \to \infty$, i.e., for every interval $(a,b) \subseteq \R$ such that $\mathbb{P}(S \in \{a,b\}) = 0$, the following holds,
\[
\lim_{N \to \infty} \mu(\{x \in X \ | \ S_N(x) \in (a,b) \}) = \mathbb{P}\left(S \in (a,b) \right).
\]
Assume that there exists a measure-preserving ergodic transformation $U \colon X \to X$ such that the pair $(T,U)$ is hyperbolic and $f \colon X \to \R$ is a $(T,U, \V)$-adapted function. Then, the same DLT holds in the mixing sense, i.e. for every pair of Borel measurable subsets $A,B \subseteq X$ and every interval $(a,b) \subseteq \R$ such that $\mathbb{P}(S \in \{a,b\}) = 0$,
\[
  \lim_{N \to \infty} \mu(\{x \in X \ | \ x \in A, \ S_N(x) \in (a,b), \ T^N x \in B \})
  = \mu(A) \cdot \mathbb{P}\left(S \in (a,b) \right) \cdot \mu(B).
\]
\end{theorem}

\begin{remark}
For a weaker conditional limit theorem that holds in a more general context see Theorem \ref{theo:metric_CCLT}. For a discussion of the case of flows see Theorems \ref{theo:metric_MCLT_flows} and \ref{theo:metric_CCLT_flows}.
\end{remark}

\subsection*{Cocycles.} Let $(X,d)$ be a metric space supporting a Borel probability measure $\mu$ and an invertible, measure-preserving transformation $T \colon X \to X$. Fix $m \in \mathbb{N}$. By an $m$-dimensional cocycle over $T$ we mean a map $C \colon X \times \mathbb{Z} \to \mathrm{GL}(m,\mathbb{R})$ satisfying the identity
\[
C(x,r+s) = C(T^r x,s) \cdot C(x,r) \quad \text{ for $\mu$-almost every } x\in X \text{ and every } r,s\in \Z\,.
\]

Denote by $\langle \cdot, \cdot\rangle$ the standard inner product on $\mathbb{R}^m$, by $\|\cdot\|$ the corresponding Euclidean norm, and by $\nu$ the induced probability measure on the projectivization $\mathbb{P}\mathbb{R}^m$. The projectivized bundle $X \times \mathbb{P}\mathbb{R}^m$ can be endowed with the product measure $\mu \otimes \nu$.

Given $(x,v) \in X \times (\mathbb{R}^m \setminus \{0\})$ or $X \times \mathbb{P}\mathbb{R}^m$ and $N \in \mathbb{N}$ consider the quantities
\begin{align*}
    \sigma(x,v,N) := \log \frac{\|C(x,N) v\|}{\|v\|} \in \mathbb{R},\\
    \sigma(x,N) := \log \sup_{v \in \mathbb{R}^m} \frac{\|C(x,N)v\|}{\|v\|} \in \mathbb{R}.
\end{align*}
We say the cocycle $C$ is log-integrable if the following maps belong to $L^1(X,\mu)$,
\[
x \mapsto \max\{0,\sigma(x,1)\}, \quad x \mapsto \max\{0,\sigma(x,-1)\}.
\]

We say the cocycle $C$ satisfies a (spatial) distributional limit theorem (DLT) on $(X \times \mathbb{PR}^m, \mu \otimes \nu)$ if there exist sequences $\A:=(A_N)_{N \in \mathbb{N}}$ of real numbers and $\V:=(V_N)_{N \in \mathbb{N}}$ of positive real numbers with $V_N \to \infty$ as $N \to \infty$, and a random variable $S$, such that the random variables 
$$
S_N(x,v) = \frac{\sigma(x,v,N) -A_N}{ V_N } \ \ \text{ on } (X\times \R^m, \mu\otimes \nu)
$$
converge in distribution to $S$ as $N \to \infty$, i.e., for every interval $(a,b) \subseteq \mathbb{R}$ such that $\mathbb{P}(S \in \{a,b\}) = 0$, the following holds,
\[
\lim_{N \to \infty} (\mu \otimes \nu)(\{(x,v) \in X \in \mathbb{PR}^m \ | \ S_N(x,v) \in (a,b) \}) = \mathbb{P}(S \in (a,b)).
\]
We usually refer to $\A$ as the averaging sequence, $\V$ as the normalizing sequence, and $S$ as the limiting distribution.

\begin{example}
The Oseledets ergodic theorem guarantees that, if $\mu$ is ergodic with respect to $T$, every log-integrable cocycle $C\colon X \times \mathbb{Z} \to \mathrm{GL}(m,\mathbb{R})$ with top Lyapunov exponent $\lambda \in \mathbb{R}$ satisfies a DLT with averaging sequence $\A$ given by $A_N := N \cdot \lambda$ for every $N \in \mathbb{N}$, normalizing sequence $\V$ given by $V_N := N$ for every $N \in \mathbb{N}$, and $S$ a (constant) random variable with distribution equal to the Dirac mass at $0\in \R$.
\end{example}

\begin{example}
Another important case in hyperbolic dynamics is the central limit theorem (CLT) for which the averaging sequence $\A$ is given by $A_N:= \lambda \cdot N$ for every $N \in \mathbb{N}$, the normalizing sequence $\V$ is given by $V_N =\smash{\sqrt{N}}$ for every $N \in \mathbb{N}$,
and the random variable $S$ has normal distribution with mean $0$. By work of Al-Saqban and Forni \cite{KZCLT}, the Kontsevich-Zorich cocycle satisfies a distributional CLT.
\end{example}

We say the cocycle $C$ is $\V$-sufficiently-bounded if for $\mu$-almost-every $x \in X$, $\nu$-almost-every $v \in \mathbb{PR}^m$, and every $N \in \mathbb{N}$,
\[
|\sigma(x,v,N) - \sigma(Tx,C(x,1)v,N)| = o_{x,v}(V_N).
\]
We say the cocycle $C$ has $\V$-simple-dominated-splitting  if for $\mu$-almost every $x \in X$, $\nu$-almost every $v,w \in \mathbb{PR}^m$, and every $N \in \mathbb{N}$, 
    \[
    |\sigma(x,v,N) - \sigma(x,w,N)| = o_{x,v,w}(V_N).
    \]

\begin{remark}
In applications, see \S\ref{S:KZ}, it is common for the cocycle $C$ to be bounded  in the following stronger sense: for $\mu$-almost-every $x \in X$ and $\nu$-almost-every $v \in \mathbb{PR}^m$, 
\[
|\sigma(x,v,1)| = O(1).
\] 
\end{remark}

\begin{remark}
We also prove below, in Lemma~\ref{lemma:exp}, that, by the Oseledets ergodic theorem, any log-integrable cocycle over an invertible ergodic dynamical system with simple top Lyapunov exponent has simple dominated splitting with respect to any sequence $\V$.
\end{remark}

Now let $U \colon X \to X$ be a measure-preserving, ergodic transformation and let
$D \colon X \times \mathbb{N} \to \mathrm{GL}(m,\R)$ be a measurable cocycle over $U$.
Recall that we say the pair $(T,U)$ is contracting if for $\mu$-almost-every $x \in X$,
\[
\lim_{n \to \infty} d(T^n U x, T^n x) = 0.
\] We say the cocycles $(C,D)$ are $(T,U, \V)$-adapted if for $\mu$-almost-every $x \in X$, $\nu$-almost every $v \in \mathbb{PR}^m$, and every $N \in \mathbb{N}$, the following estimate holds,
\[
|\sigma(x,v,N) - \sigma(Ux,D(x,1)v,N)| = o_{x,v}(V_N).
\]

\begin{remark}
In applications, see \S\ref{S:KZ}, it is common to have the following stronger condition: for $\mu$-almost-every $x \in X$, $\nu$-almost-every $v \in \mathbb{PR}^m$, and every $N \in \mathbb{N}$,
\[
|\sigma(x,v,N) - \sigma(Ux,D(x,1)v,N)| = O_{x,v}(1).
\]
\end{remark}

The following is the main result of this paper for cocycles.

\begin{theorem}
\label{theo:cocycle_MCLT_intro}
Let $(X,d)$ be a metric space supporting a Borel probability measure $\mu$, let $T \colon X \to X$ be an invertible, measure-preserving transformation, and let $C \colon X \times \mathbb{Z} \to \mathrm{GL}(m,\mathbb{R})$ be a measurable cocycle over $T$. Assume that $C$ satisfies a DLT on $(X \times \mathbb{PR}^m, \mu \otimes \nu)$ with averaging sequence $\A := (A_N)_{N \in \mathbb{N}}$, normalizing sequence $\V=(V_N)_{N \in \mathbb{N}}$, and limiting distribution $S$ in the sense that the random variables
\[
S_N(x,v) := \frac{\sigma(x,v,N) - A_N}{V_N}
\quad \text{ on } \ (X\times \mathbb{P}\R^m, \mu\otimes \nu)
\]
converge in distribution to  $S$ as $N \to \infty$, i.e., for every interval $(a,b) \subseteq \R$ such that $\mathbb{P}(S \in \{a,b\}) = 0$, the following holds,
\[
\lim_{N \to \infty} (\mu \otimes \nu)( \{(x,v) \in X \times \mathbb{PR}^m \ | \ S_N(x) \in (a,b)\}) = \mathbb{P}\left(S \in (a,b) \right).
\]
Assume that $C$ is $\V$-sufficiently-bounded and has $\V$-simple-dominated-splitting. Assume in addition that there exist an ergodic transformation $U \colon X \to X$ and a measurable cocycle $D \colon X \times \mathbb{N} \to \mathrm{GL}(m,\mathbb{R})$ over $U$ such that $(T,U)$ is contracting and $(C,D)$ is $(T,U, \V)$-adapted. Then, the above DLT holds in the mixing sense, i.e. for every pair of Borel measurable subsets $A,B \subseteq X$ and every interval $(a,b) \subseteq \R$ such that $\mathbb{P}(S \in \{a,b\}) = 0$,
\begin{gather*}
  \lim_{N \to \infty}(\mu \otimes \nu) (\{ (x,v) \in X \times \mathbb{PR}^m \ | \ x \in A, \ S_N(x,v) \in (a,b), \ T^N x \in B\})\\
  = \mu(A) \cdot \mathbb{P}\left(S \in (a,b) \right) \cdot \mu(B).
\end{gather*}
\end{theorem}


\begin{remark}
For a weaker conditional limit theorem that holds in a more general context see Theorem \ref{theo:cocycle2_CCLT}. For a discussion of the case of flows see Theorems \ref{theo:cocycle_MCLT_flows} and \ref{theo:cocycle2_CCLT_flows}.
\end{remark}

\begin{remark}
For a discussion of central limit theorems for the operator norm of cocycles see Theorems \ref{theo:cocycle_CCLT_norm} and \ref{theo:cocycle_CCLT_norm2}. For analogous results for flows see Theorems \ref{theo:cocycle_CCLT_norm_flows}
 and \ref{theo:cocycle_CCLT_norm2_flows}.
\end{remark}

\begin{remark}
    For a discussion of central limit theorems for ``generic" sections of cocycles see Theorems \ref{theo:cocycle_CCLT_sec} and \ref{theo:cocycle_CCLT_sec2}. For analogous results for flows see Theorems \ref{theo:cocycle_CCLT_sec_flows}
 and \ref{theo:cocycle_CCLT_sec2_flows}.
\end{remark}

\subsection*{The Kontsevich-Zorich cocycle.} The Kontsevich--Zorich (KZ) cocycle, introduced in  \cite{Kon97,KZ97}, is arguably the central object of study in Teichm\"uller dynamics.  As an application of the results discussed above, we prove mixing laws of large numbers and mixing central limit theorems for subbundles of exterior powers of the Kontsevich-Zorich cocycle over $\mathrm{SL}(2,\mathbb{R})$-invariant suborbifolds of Abelian differentials under natural, well studied conditions. We prove mixing limit theorems in the spirit of Theorem \ref{theo:cocycle_MCLT_intro} but also for the operator norm and for ``generic" sections of the KZ cocycle; see Theorems \ref{theo:LLN1},\ref{theo:LLN2},\ref{theo:LLN3},\ref{theo:CLT1},\ref{theo:CLT2},\ref{theo:CLT3} for precise statements. In the case of mixing laws of large numbers, the starting point is the Oseledets ergodic theorem. In the case of mixing central limit theorems, the starting point is the central limit theorem for the KZ cocycle proved by Al-Saqban and Forni in \cite{KZCLT}.

As a particularly important application of these results we consider the case of the invariant part of the full KZ cocycle over loci of orientation double covers of quadratic differentials. This application relies crucially on work of Filip \cite{Fil17} and Bell, Delecroix, Gadre, Gutiérrez-Romo, and Schleimer \cite{flow_group} to verify the appropriate conditions needed for the mixing limit theorems to hold. These theorems are crucial for the applications to the study of the arithmetic/homological complexity of long simple closed geodesics on negatively curved surfaces in forthcoming work of Arana-Herrera and Honaryar \cite{statistics}.

\subsection*{Organization of the paper.} In \S2 we show how to upgrade limit theorems for Birkhoff sums over dynamical systems to mixing limit theorems; the case of flows is also considered. In \S3 we prove analogous results for cocycles over dynamical systems. In \S4 we discuss important applications of the results in \S3 to the Kontsevich-Zorich cocycle.

\subsection*{Acknowledgements.} The authors would like to thank Pouya Honaryar for introducing them to the subject of mixing limit theorems. The authors would also like to thank Dmitry Dolgopyat for very enlightening conversations on the subject of this paper. This material is based upon work funded by the National Science Foundation: the second author is supported by grants DMS 1600687 and DMS 2154208.

\section{Birkhoff sums}
\label{S:Birk}

\subsection*{Outline of this section.} In this section we discuss how to upgrade distributional limit theorems for Birkhoff sums of dynamical systems to mixing distributional limit theorems under mild ergodicity and hyperbolicity assumptions; see Theorem \ref{theo:metric_MCLT} for the main result of this section. We begin by proving different conditional distributional limit theorems for Birkhoff sums, see Theorem \ref{theo:metric_CCLT} and Corollary \ref{cor:tech}, and then use these results to prove the main theorem. Although the results in this section are not directly used in applications, their proofs inspire proofs in later sections.

\subsection*{Statement of the main result.} We begin by reviewing the statement of our main result for Birkhoff sums. Let $(X,d)$ be a metric space supporting a Borel probability measure $\mu$, an invertible, measure-preserving transformation $T\colon X \to X$, and a measure-preserving, ergodic transformation $U \colon X \to X$.

Recall we say the pair $(T,U)$ is contracting if for $\mu$-almost-every $x \in X$, 
\[
\lim_{n \to \infty} d(T^n U x, T^n x) = 0.
\]

Let $\V := (V_N)_{N \in \mathbb{N}}$ be a sequence of positive real numbers with $V_N \to \infty$ as $N \to \infty$. Recall that we say a function $f \colon X \to \R$ is $(T,U, \V)$-adapted if for $\mu$-almost-every $x \in X$ and for every $N\in \N$, the following holds,
\[
\sum_{n=0}^N |f(T^n U x) - f(T^n x)| = o_x(V_N).
\]

The following is the main result of this section; compare to Theorem \ref{theo:metric_MCLT_intro}.

\begin{theorem}
\label{theo:metric_MCLT}
Let $(X,d)$ be a metric space supporting a Borel probability measure $\mu$, let $T\colon X \to X$ be an invertible, measure-preserving transformation, and let $f \colon X \to \mathbb{R}$ be a bounded, measurable function. Assume that the Birkhoff sums of $f$ with respect to $T$ satisfy a DLT on $(X,\mu)$ with averaging sequence $\A :=(A_N)_{N \in \mathbb{N}}$, normalizing sequence $\V :=(V_N)_{N \in \mathbb{N}}$, and limiting distribution $S$ in the sense that the random variables
\[
S_N(x) := \frac{\sum_{n=0}^{N-1}f(T^n x) - A_N}{V_N} \ \ \text{on } (X,\mu)
\]
converge in distribution to $S$ as $N \to \infty$, i.e., for every interval $(a,b) \subseteq \R$ such that $\mathbb{P}(S \in \{a,b\}) = 0$, the following holds,
\[
\lim_{N \to \infty} \mu(\{x \in X \ | \ S_N(x) \in (a,b) \}) = \mathbb{P}\left(S \in (a,b) \right).
\]
Assume that there exists a measure-preserving ergodic transformation $U \colon X \to X$ such that the pair $(T,U)$ is contracting and $f \colon X \to \R$ is a $(T,U, \V)$-adapted function. Then, the same DLT holds in the mixing sense, i.e. for every pair of Borel measurable subsets $A,B \subseteq X$ and every interval $(a,b) \subseteq \R$ such that $\mathbb{P}(S \in \{a,b\}) = 0$,
\[
  \lim_{N \to \infty} \mu (\{x \in X \ | \ x \in A, \ S_N(x) \in (a,b),  \ T^N x \in B \})
  = \mu(A) \cdot \mathbb{P}\left(S \in (a,b) \right) \cdot \mu(B).
\]
\end{theorem}

\subsection*{A conditional distributional limit theorem.} To prove Theorem \ref{theo:metric_MCLT} we first prove the following result of independent interest.

\begin{theorem}
\label{theo:metric_CCLT}
Let $(X,\mathcal{B},\mu)$ be a probability space, let $T\colon X \to X$ be a measure-preserving, ergodic transformation, and let $f \colon X \to \R$ be a bounded, measurable function. Suppose the Birkhoff sums of $f$ with respect to $T$ satisfy a DLT on $(X,\mathcal{B},\mu)$ with averaging sequence $\A :=(A_N)_{N \in \mathbb{N}}$, normalizing sequence $\V :=(V_N)_{N \in \mathbb{N}}$, and limiting distribution $S$ in the sense that the random variables
\[
S_N(x) := \frac{\sum_{n=0}^{N-1}f(T^n x) - A_N}{V_N} \ \  \text{on } (X,\mathcal{B},\mu)
\]
converge in distribution to $S$ as $N \to \infty$, i.e., for every interval $(a,b) \subseteq \R$ such that $\mathbb{P}(S \in \{a,b\}) = 0$, the following holds,
\[
\lim_{N \to \infty} \mu(\{x \in X \ | \ S_N(x) \in (a,b) \}) = \mathbb{P}\left(S \in (a,b) \right).
\]
Then, the same DLT holds in the conditional sense, i.e. for every $A \in \mathcal{B}$ and every interval $(a,b) \subseteq \R$ such that $\mathbb{P}(S \in \{a,b\}) = 0$,
\[
  \lim_{N \to \infty} \mu(\{x \in X \ | \ x \in A, \ S_N(x) \in (a,b)\})
  = \mu(A) \cdot \mathbb{P}\left(S \in (a,b) \right).
\]
\end{theorem}

Before proving Theorem \ref{theo:metric_CCLT}, we first review some basic terminology. Let $(\Omega,\mathcal{B},\mathbb{P})$ be a probability space and $(X_n)_{n \in \N}$ be a sequence of integrable random variables on it. We say $X_n$ converges weakly in $L^1$ to an integrable random variable $X$ on $(\Omega,\mathcal{B},\mathbb{P})$ if for every bounded random variable $Y$ on $(\Omega,\mathcal{B},\mathbb{P})$ the following holds,
\[
\lim_{n \to \infty} \mathbb{E}(X_n Y) = \mathbb{E}(XY).
\]
Equivalently, $X_n$ converges weakly in $L^1$ to $X$ if for every measurable set $A \in \mathcal{B}$,
\[
\lim_{n \to \infty} \mathbb{E}\left(X_n \chi_A\right) = \mathbb{E}(X\chi_A).
\]

To reduce to a question about weak $L^1$ convergence, we prove the following general lemma; compare to \cite[Theorem 1]{MCLT76}.

\begin{lemma}
    \label{lemma:tech}
    Let $(\Omega,\mathcal{B},\mathbbm{P})$ be a probability space and $(X_n)_{n \in \N}$ be a sequence of integrable random variables on it. Then $X_n$ converges conditionally in distribution to an integrable random variable $X$ in the sense that for every measurable set $A \in \mathcal{B}$ and every interval $(a,b) \subseteq \mathbb{R}$ with $\mathbb{P}(X \in \{a,b\}) = 0$,
    \[
    \lim_{n \to \infty} \mathbb{P}(A \cap X_n \in (a,b)) =  \mathbb{P}(A) \cdot \mathbb{P}(X \in (a,b)) 
    \]
    if and only for every fixed $t \in \R$,
    \[
    e^{itX_n} \to \mathbb{E}(e^{itX}) \ \ \text{weakly in }L^1 \text{ as }n \to \infty.
    \]
\end{lemma}

\begin{proof}
    For every $A \in \mathcal{B}$ with $\mathbb{P}(A) > 0$ consider the probability space $(A,\mathcal{B}|_A,\mathbb{P}|_A)$ and the integrable random variables $(X_n|_A)_{n \in \N}$ on it. Notice that $X_n$ converges conditionally in distribution to $X$ if and only if $(X_n|_A)_{n \in \N}$ converges in distribution to $X$ for every $A \in \mathcal{B}$ with $\mathbb{P}(A) > 0$. By L\'evy's continuity theorem, this condition is equivalent to the following convergence of characteristic functions for every $t \in \R$ :
    \[
    \mathbb{E}(e^{it X_n}|A) \to \mathbb{E}(e^{it X}) \ \ \text{as } n \to \infty.
    \]
    But this is exactly the condition that for every measurable set $A \in \mathcal{B}$ with $\mathbb{P}(A) > 0$ and every $t \in \mathbb{R}$, 
    \[
    \mathbb{E}(\chi_A e^{itX_n}) \to \mathbb{E}(\chi_A\mathbb{E}(e^{itX}))  \ \ \text{as } n \to \infty.
    \]
    Thus, we see that the original condition is equivalent to
    \[
    e^{itX_n} \to \mathbb{E}(e^{itX}) \ \ \text{weakly in }L^1 \text{ as } n \to \infty\text{ for every }t\in \R. \qedhere
    \]
\end{proof}

We are now ready to prove Theorem \ref{theo:metric_CCLT}

\begin{proof}[Proof of Theorem \ref{theo:metric_CCLT}]
    Consider the corrected Birkhoff sums $S_N$ as random variables on $(X,\mathcal{B},\mu)$. By Lemma \ref{lemma:tech}, it is enough to check that  for every fixed $t \in \R$,
    \[
    e^{itS_N} \to \mathbb{E}(e^{itS}) \ \ \text{weakly in }L^1 \text{ as } N \to \infty.
    \]

    For the rest of this discussion we fix $t \in \R$ and prove this statement. By the Dunford-Pettis theorem, a sequence of integrable random variables is sequentially $L^1$ weakly compact if and only if it is uniformly integrable.  
    In particular, as the random variables $e^{itS_N}$ are uniformly bounded, it is enough to show that the only weak $L^1$ limit point of $(e^{itS_N})_{N \in \N}$ is the constant random variable $\mathbb{E}(e^{itS})$.
    
    Let $Y$ be a weak $L^1$ limit point of this sequence along times $\{N_k\}_{k \in \N}$. We claim that $Y = Y \circ T$ almost surely. Indeed, because $f$ is bounded, 
     it follows directly from the definition of $S_N$ and the dominated convergence theorem that, 
    \[
    \mathbb{E}(|e^{itS_N} \circ T - e^{itS_N}|) \to 0 \ \ \text{as }N \to \infty.
    \]
    The Hölder inequality then implies that
    \[
    e^{itS_{N_k}} \circ T \to Y \ \ \text{weakly in }L^1 \text{ as } k \to \infty.
    \]
    But, at the same time, the $T$-invariance of $\mu$ implies
    \[
    e^{itS_{N_k}} \circ T \to Y \circ T \ \ \text{weakly in }L^1 \text { as } k \to \infty.
    \]
    As weak $L^1$ limits are unique almost surely, the claim follows.
    
    To conclude we notice that, as $T$ acts ergodically, $Y$ must be constant almost everywhere. Pairing $e^{itS_{N_k}}$ against the constant function $1$ and using the assumed DLT reveals that $Y$ must be $\mathbb{E}(e^{itS})$. This finishes the proof.
\end{proof}

The proof of Theorem \ref{theo:metric_MCLT} will use Theorem \ref{theo:metric_CCLT} in the following form.

\begin{corollary}
\label{cor:tech}
Let $(X,\mathcal{B},\mu)$ be a probability space, let $T\colon X \to X$ be a measure-preserving, ergodic transformation, and let $f \colon X \to \R$ be a bounded, measurable function. Suppose the Birkhoff sums of $f$ with respect to $T$ satisfy a DLT $(X,\mathcal{B},\mu)$ with averaging sequence $\A :=(A_N)_{N \in \mathbb{N}}$, normalizing sequence $\V :=(V_N)_{N \in \mathbb{N}}$, and limiting distribution $S$ in the sense that the random variables
\[
S_N(x) := \frac{\sum_{n=0}^{N-1}f(T^n x) - A_N}{V_N} \ \  \text{on } (X,\mathcal{B},\mu)
\]
converge in distribution to $S$ as $N \to \infty$, i.e., for every interval $(a,b) \subseteq \R$ such that $\mathbb{P}(S \in \{a,b\}) = 0$, the following holds,
\[
\lim_{N \to \infty} \mu(\{x \in X \ | \ S_N(x) \in (a,b) \}) = \mathbb{P}\left(S \in (a,b) \right).
\]
Then, the same DLT holds in the following stronger sense: for every $B \in \mathcal{B}$ and every interval $(a,b) \subseteq \R$ such that $\mathbb{P}(S \in \{a,b\}) = 0$,
\[
  \lim_{N \to \infty} \mu(\{ x \in X \ | \ S_N(x) \in (a,b), \ T^N x \in B\}) = \mathbb{P}\left(S \in (a,b) \right) \cdot \mu(B).
\]
Furthermore, if $X$ is a metric space and $\mathcal{B}$ is its Borel $\sigma$-algebra, then, for every bounded Lipschitz function $\phi \colon X \to \R$ and every $t \in \R$,
\[
  \lim_{N \to \infty} \int_X e^{itS_N(x)} \cdot \phi(T^N x)  \thinspace d\mu(x)
  = \mathbb{E}(e^{itS}) \cdot \mu(\phi).
\]
\end{corollary}

To prove Corollary \ref{cor:tech} we first prove the following technical lemma.

\begin{lemma}
\label{lemma:inverting}
Let $(X,\mathcal{B},\mu)$ be a probability space, let $T\colon X \to X$ be a measure-preserving transformation, and let $f \colon X \to \R$ be a measurable function. Suppose that the Birkhoff sums of $f$ with respect to $T$ satisfy a DLT on $(X,\mathcal{B},\mu)$ with averaging sequence $\A :=(A_N)_{N \in \mathbb{N}}$, normalizing sequence $\V :=(V_N)_{N \in \mathbb{N}}$, and limiting distribution $S$ in the sense that the random variables
\[
S_N(x) := \frac{\sum_{n=0}^{N-1}f(T^n x) - A_N}{V_N} \ \  \text{on } (X,\mathcal{B},\mu)
\]
converge in distribution to $S$ as $N \to \infty$, i.e., for every interval $(a,b) \subseteq \R$ such that $\mathbb{P}(S \in \{a,b\}) = 0$, the following holds,
\[
\lim_{N \to \infty} \mu(\{x \in X \ | \ S_N(x) \in (a,b) \}) = \mathbb{P}\left(S \in (a,b) \right).
\]
Then, the same DLT  holds for Birkhoff sums of $f$ with respect to $T^{-1}$, i.e., the random variables 
\[
S_N^{-1}(x) := \frac{\sum_{n=0}^{N-1}f(T^{-n} x) - A_N}{V_N} \ \  \text{on } (X,\mathcal{B},\mu)
\]
converge in distribution to $S$ as $N \to \infty$, i.e., for every interval $(a,b) \subseteq \R$ such that $\mathbb{P}(S \in \{a,b\}) = 0$, the following holds
\[
\lim_{N \to \infty} \mu(\{x \in X \ | \ S_N^{-1}(x) \in (a,b) \}) = \mathbb{P}\left(S \in (a,b) \right).
\]
\end{lemma}

\begin{proof}
    Let $(a,b) \subseteq \R$ be an interval such that $\mathbb{P}(S \in \{a,b\}) = 0$. Notice that, for every $x \in X$ and every $N \in \N$, 
    \[
    S_N(x) = S_N^{-1}(T^{N-1}(x)).
    \]
    In particular, for every $N \in \N$,
    \[
    \mu(\{x \in X \ | \ S_N^{-1}(x) \in (a,b)\}) = \mu(\{x \in X \ | \ 
    S_N(T^{-(N-1)}x) \in (a,b)\}).
    \]
    The $T$-invariance of $\mu$ guarantees that, for every $N \in \N$,
    \[
    \mu(\{x \in X \ | \ 
    S_N(T^{-(N-1)}x) \in (a,b)\}) = \mu(\{x \in X \ | \ 
    S_N(x) \in (a,b)\}).
    \]
    The result now follows from the fact that $S_N$ converges in distribution to $S$ as $N \to \infty$.
\end{proof}

We are now ready to prove Corollary \ref{cor:tech}.

\begin{proof}[Proof of Corollary \ref{cor:tech}]
    Without loss of generality consider $B \in \mathcal{B}$ with $\mu(B) > 0$. The $T$-invariance of $\mu$ ensures that, for every $N \in \N$,
    \[
    \mu(T^{-N} B) = \mu(B) > 0.
    \]
    For every $N \in \N$ consider the probability space $(T^{-N} B,\mathcal{B}|_{T^{-N}B},\mu|_{T^{-N} B} )$ and the random variable $S_N|_{T^{-N} B}$ on it. Our goal is to show that $(S_N|_{T^{-N}B})_{N \in \N}$ converges in distribution to $S$. By L\'evy's continuity theorem, this condition is equivalent to the following convergence of characteristic functions for every $t \in \R$ :
    \[
    \mathbb{E}(e^{itS_N} | T^{-N}B) \to \mathbb{E}(e^{itS}) \ \ \text{as }N \to \infty.
    \]
    For the rest of this discussion we fix $t \in \R$ and show that
    \[
    \lim_{N \to \infty} \int_X e^{itS_N(x)} \cdot \chi_B(T^N x) \thinspace d\mu(x) = \mathbb{E}(e^{itS}) \cdot \mu(B).
    \]
    
    For $x \in X$ and $N \in \N$ consider the corrected Birkhoff sum of $f$ with respect to $T^{-1}$,
    \[
    S_N^{-1}(x) := \frac{\sum_{n=0}^{N-1}f(T^{-n}x) -A_N}{V_N}.
    \]
    By definition, for every $x \in X$ and every $N \in \N$,
    \[
    S_N(x) = S_N^{-1}(T^{N-1} x).
    \]
    Hence, for every $N \in \mathbb{N}$ we can write
    \begin{gather*}
    \int_X e^{itS_N(x)} \cdot \chi_B(T^N(x)) \thinspace d\mu(x) 
    = \int_X  e^{it S_N^{-1}(T^{N-1}(x))} \cdot \chi_B(T^N x) \thinspace d\mu(x).
    \end{gather*}
    The $T$-invariance of $\mu$ then ensures that, for every $N \in \mathbb{N}$,
    \begin{gather*}
    \int_X e^{it S_N^{-1}(T^{N-1}(x))} \cdot \chi_B(T^N x) \thinspace d\mu(x) 
    = \int_X  e^{it S_N^{-1}(x)} \cdot \chi_{T^{-1}B}(x) \thinspace d\mu(x).
    \end{gather*}
    It remains to show that
    \[
    \lim_{N \to \infty} \int_X e^{it S_N^{-1}(x)} \cdot \chi_{T^{-1}B}(x) \thinspace d\mu(x) = \mathbb{E}(e^{itS}) \cdot \mu(T^{-1}B)=\mathbb{E}(e^{itS}) \cdot \mu(B).
    \]
   This follows directly from Theorem \ref{theo:metric_CCLT} and Lemma~\ref{lemma:inverting}.

    The fact that, in the case where $X$ is a metric space and $\mathcal{B}$ is its Borel $\sigma$-algebra, one can replace the characteristic function $\chi_B$ with an arbitrary bounded Lipschitz function $\phi \colon X \to \R$ follows by standard approximation arguments.
 \end{proof}



\subsection*{Mixing.} To prove Theorem \ref{theo:metric_MCLT} we will also use the following mixing result.

\begin{theorem}
    \label{theo:mix}
    Let $(X,d)$ be a metric space supporting a Borel probability measure $\mu$, an invertible, measure-preserving transformation $T\colon X \to X$, and a measure-preserving, ergodic transformation $U \colon X \to X$. Suppose $(T,U)$ is contracting. Then $T$ is mixing.
\end{theorem}

\begin{proof}
    By standard approximation arguments, it is enough to show that for every Borel measurable subset $A \subseteq X$ and every bounded Lipschitz function $\phi \colon X \to \R$,
    \[
    \lim_{N \to \infty} \int_X \chi_A(x) \cdot \phi(T^N x) \thinspace d\mu(x) = \mu(A) \cdot \mu(\phi).
    \]
    This is equivalent to showing that
    \[
    \phi \circ T^N \to \mu(\phi) \ \ \text{weakly in }L^1 \text{ as }N \to \infty.
    \]
   By the Dunford-Pettis theorem theorem, a sequence of integrable functions is sequentially $L^1$ weakly compact if and only if it is uniformly integrable.    In particular, as the functions $\phi \circ T^N$ are uniformly bounded, it is enough to show that the only weak $L^1$ limit point of $(\phi \circ T^N)_{N \in \N}$ is the constant function $\mu(\phi)$.
    
    Let $\varphi$ be a weak $L^1$ limit point of this sequence along times $\{N_k\}_{k \in \N}$. We claim that $\varphi = \varphi \circ U$ almost surely. Indeed, it follows directly from the fact that $(T,U)$ is hyperbolic, the fact that $\phi$ is bounded Lipschitz, and the dominated convergence theorem that, 
    \[
    \mathbb{E}(|\phi \circ T^N \circ U - \phi \circ T^N|) \to 0 \ \ \text{as }N \to \infty.
    \]
    The Hölder inequality then implies that
    \[
    \phi \circ T^N \circ U \to \varphi \ \ \text{weakly in }L^1 \text{ as }N \to \infty.
    \]
    But, at the same time, the $U$-invariance of $\mu$ implies
    \[
    \phi \circ T^N \circ U \to \varphi \circ U \ \ \text{weakly in }L^1 \text{ as }N \to \infty
    \]
    As weak $L^1$ limits are unique almost surely, the claim follows.
    
    To conclude notice that, as $U$ acts ergodically, $\varphi$ must be constant almost everywhere. Pairing $\phi \circ T^N$ against the constant function $1$ and using the $T$-invariance of $\mu$ reveals that $\varphi$ must be the constant function $\mu(\phi)$. This finishes the proof.
\end{proof}

\subsection*{Proof of the main result.} We are now ready to prove Theorem \ref{theo:metric_MCLT}.

\begin{proof}[Proof of Theorem \ref{theo:metric_MCLT}]
    Without loss of generality consider a pair of Borel measurable sets $A,B \subseteq X$ with $\mu(A),\mu(B) > 0$. Theorem \ref{theo:mix} ensures $T$ is mixing. In particular,
    \[
    \lim_{N \to \infty} \mu(A \cap T^{-N} B)  = \mu(A) \cdot \mu(B) > 0.
    \]
    For every $N \in \N$ sufficiently large consider the probability space $$(A \cap T^{-N} B,\mathcal{B}|_{A \cap T^{-N}B},\mu|_{A \cap T^{-N} B} )$$ and the random variable $S_N|_{A \cap T^{-N} B}$ on it. We aim to prove that $(S_N|_{A \cap T^{-N}B})_{N \in \N}$ converges in distribution to $S$ as $N \to \infty$. By L\'evy's continuity theorem, this condition is equivalent to the following convergence of characteristic functions for every $t \in \R$ :
    \[
    \mathbb{E}(e^{itS_N} | A \cap T^{-N}B) \to \mathbb{E}(e^{itS}) \ \ \text{as }N \to \infty.
    \]
    As $T$ is mixing, this is equivalent to showing that, for every $t \in \R$,
    \[
    \lim_{N \to \infty} \int_X \chi_A(x) \cdot e^{itS_N(x)} \cdot \chi_B(T^Nx) \thinspace d\mu(x) = \mu(A) \cdot \mathbb{E}(e^{itS}) \cdot \mu(B).
    \]
    Standard approximation arguments show that one can replace $\chi_B$ in this statement by a bounded Lipschitz function $\phi \colon X \to \R$. We do so for the rest of this discussion. Now consider for fixed $t \in \R$ the random variables
    \[
    F_N := e^{itS_N(x)} \cdot \phi(T^N x), \ \ N \in \N.
    \]
    Our goal is to show that 
    \[
    F_N \to \mathbb{E}(e^{itS}) \cdot \mu(\phi) \ \ \text{weakly in }L^1 \text{ as } N \to \infty.
    \]
    By the Dunford-Pettis theorem, a sequence of integrable random variables is sequentially $L^1$ weakly compact if and only if it is uniformly integrable. In particular, as the random variables $F_N$ are uniformly bounded, it is enough to show that the only weak $L^1$ limit point of $(F_N)_{N \in \N}$ is the constant random variable $\mathbb{E}(e^{itS}) \cdot \mu(\phi)$.
    
    Let $F$ be a weak $L^1$ limit point of this sequence along  times $\{N_k\}_{k \in \N}$. We claim that $F = F \circ U$ almost surely. Indeed, it follows directly from the definition of $F_N$, the fact that $(T,U)$ is hyperbolic, the fact that $f$ is bounded and $(T,U, \V)$-adapted, the fact that $\phi$ is bounded Lipschitz, and the dominated convergence theorem, that
    \[
    \mathbb{E}(|F_N \circ U - F_N|) \to 0 \ \ \text{as }N \to \infty.
    \]
    The Hölder inequality then implies that
    \[
    F_{N_k} \circ U \to F \ \ \text{weakly in }L^1 \text{ as } k \to \infty.
    \]
   At the same time, the $U$-invariance of $\mu$ implies
    \[
    F_{N_k} \circ U \to F \circ U \ \ \text{weakly in }L^1 \text{ as } k \to \infty.
    \]
    As weak $L^1$ limits are unique almost surely, the claim follows.
    
    To conclude notice that, as $U$ acts ergodically, $F$ must be constant almost everywhere. Pairing $F_N$ against the constant function $1$ and using Corollary \ref{cor:tech} reveals that $F$ must be $\mathbb{E}(e^{itS}) \cdot \mu(\phi)$. This finishes the proof.
\end{proof}

\subsection*{Flows.} We now discuss the case of flows. Let $(X,d)$ be a metric space supporting a Borel probability measure $\mu$ and a measure-preserving, ergodic flow $A := \{a_t \colon X \to X\}_{t \in \mathbb{R}}$.

Given a measurable function $f\colon X \to \R$, the ergodic integrals of $f$ with respect to $A$ satisfy a (spatial) distributional limit theorem (DLT) on $(X, \mu)$ if there exists a real function $\A:=(A_T)_{T \in \R}$, a positive real function $\V:=(V_T)_{T \in \mathbb{R}}$ with $V_T \to \infty$ as $T \to \infty$, and a random variable $S$, such that the random variables
\[
S_T(x) := \frac{\int_0^T  f(a_t x) \thinspace dt- A_T}{V_T} \ \ \text{on } (X,\mu)
\]
converge in distribution to $S$ as $N \to \infty$, that is, for every interval $(a,b) \in \mathbb{R}$ with $\mathbb{P}(S \in \{a,b\}) = 0$, the following holds,
\[
\lim_{N \to \infty} \mu(\{x \in X \ | \ S_T(x) \in (a,b)\}) = \mathbb{P}(S \in (a,b)).
\]
We usually refer to $\A$ as the averaging function, $\V$ as the normalizing function, and $S$ as the limiting distribution.

Let $U \colon X \to X$ be a measure-preserving, ergodic transformation. We say the pair $(A,U)$ is contracting if for $\mu$-almost-every $x \in X$,
\[
\lim_{t \to \infty} d(a_t U x, a_t x) = 0.
\]
We say a function $f \colon X \to \R$ is $(A,U, \V)$-adapted if for $\mu$-almost-every $x \in X$ and every $T > 0$, the following holds,
\[
\int_0^T |f(a_t U x) - f(a_t x)| \thinspace dt = o_x(V_T).
\]

\begin{remark}
The pair $(A,U)$ is called strongly contracting if for $\mu$-almost-every $x \in X$,
\[
\int_0^{+\infty}  d(a_t U x, a_tx) <+ \infty\,
\]
and hyperbolic if for $\mu$-almost-every $x \in X$ there exist $C >0$ and $\kappa > 0$ such that
\[
 d(a_t U x, a_tx) \leq C e^{-\kappa t} \ \ \text{ for all } t\geq 0\,.
\]
If the pair $(A,U)$ is called strongly contracting, then any Lipschitz function $f \colon X \to \R$ is $(A,U, \V)$-adapted for any diverging function $\V=(V_T)_{T \in \mathbb{R}}$. Furthermore, for $\mu$-almost-every $x \in X$ and every $T \geq 0$,
\[
\int_0^T |f(a_t U x) - f(a_t x)| \thinspace dt = O_x(1).
\]
\end{remark}

The following is the main result for flows; its proof is analogous to that of Theorem \ref{theo:metric_MCLT} and its details are left to the reader.

\begin{theorem}
\label{theo:metric_MCLT_flows}
Let $(X,d)$ be a metric space supporting a Borel probability measure $\mu$, let $A := \{a_t \colon X \to X\}_{t \in \mathbb{R}}$ be a measure-preserving, ergodic flow, and let $f \colon X \to \mathbb{R}$ a bounded, measurable function. Assume that the Birkhoff integrals of $f$ with respect to $A$ satisfy a DLT on $(X,\mu)$ with averaging function $\A := (A_T)_{T \in \mathbb{R}}$,
normalizing function $\V=(V_T)_{T \in \mathbb{R}}$, and limiting distribution $S$, i.e., the random variables
\[
S_T(x) := \frac{\int_0^T f(a_t x)  \thinspace dt - A_T}{V_T} \ \ \text{on } (X,\mu)
\]
converge in distribution to $S$ as $T \to \infty$, i.e., for every interval $(a,b) \subseteq \R$ such that $\mathbb{P}(S \in \{a,b\}) = 0$, the following holds
\[
 \lim_{T \to \infty} \mu(\{x \in X \ | \ S_T(x) \in (a,b)\}) = \mathbb{P}\left(S \in (a,b) \right).
\]
Assume there exists a measure-preserving, ergodic transformation $U \colon X \to X$ such that the pair $(A,U)$ is hyperbolic and such that the function $f \colon X \to \R$ is $(A,U, \V)$-adapted. Then, the same DLT holds in the mixing sense, i.e. for every pair of Borel measurable subsets $B,C \subseteq X$ and every interval $(a,b) \subseteq \R$ such that $\mathbb{P}(S \in \{a,b\}) = 0$,
\[
 \lim_{T \to \infty} \mu(\{x \in X \ | \ x \in B, \ S_T(x) \in (a,b), \ a_T x \in C\})
  = \mu(B) \cdot \mathbb{P}\left(S \in (a,b) \right) \cdot \mu(C).
\]
\end{theorem}

We also highlight the following result of independent interest; its proof is analogous to that of Theorem \ref{theo:metric_CCLT} and its details are left to the reader.

\begin{theorem}
\label{theo:metric_CCLT_flows}
Let $(X,\mathcal{B},\mu)$ be a probability space, let $A := \{a_t \colon X \to X\}_{t \in \mathbb{R}}$ be a measure-preserving, ergodic flow, and let $f \colon X \to \R$ be a bounded, measurable function. Suppose the Birkhoff integrals of $f$ with respect to $A$ satisfy a DLT on $(X,\mu)$ with averaging function $\A := (A_T)_{T \in \mathbb{R}}$,
normalizing function $\V=(V_T)_{T \in \mathbb{R}}$, and limiting distribution $S$, i.e., the random variables
\[
S_T(x) := \frac{\int_0^T f(a_t x) \thinspace dt - A_T}{V_T} \ \ \text{on } (X,\mu)
\]
converge in distribution to $S$ as $N \to \infty$, i.e., for every interval $(a,b) \subseteq \R$ such that $\mathbb{P}(S \in \{a,b\}) = 0$, the following holds
\[
\lim_{T \to \infty} \mu(\{x \in X \ | \ S_T(x) \in (a,b)\}) = \mathbb{P}\left(S \in (a,b) \right).
\]
Then, the same DLT holds in the conditional sense, i.e. for every $B \in \mathcal{B}$ and every interval $(a,b) \subseteq \R$ such that $\mathbb{P}(S \in \{a,b\}) = 0$,
\[
  \lim_{T \to \infty} \mu(\{x \in X \ | \ x \in B, \ S_T(x) \in (a,b) \})
  = \mu(B) \cdot \mathbb{P}\left(S \in (a,b) \right).
\]
\end{theorem}

\section{Cocycles}

\subsection*{Outline of this section.} In this section we discuss how to upgrade distributional limit theorems for cocycles over dynamical systems to mixing distributional limit theorems under mild ergodicity and hyperbolicity assumptions; see Theorem \ref{theo:cocycle_MCLT} for the main result of this section. We begin by proving different conditional central limit theorems for cocycles, see Theorem \ref{theo:cocycle2_CCLT} and Corollary \ref{cor:tech_cocycle}, and then use these results to prove the main theorem. The proofs are inspired by those for Birkhoff sums in \S\ref{S:Birk} but require extra ingredients. Lemma \ref{lemma:exp} is of particular importance for our approach. This lemma is later used to deduce other distributional limit theorems; see Theorems \ref{theo:cocycle_CCLT_norm}, \ref{theo:cocycle_CCLT_norm2}, \ref{theo:cocycle_CCLT_sec}, \ref{theo:cocycle_CCLT_sec2}. The section concludes with a brief discussion on similar results for flows. 

\subsection*{Statement of the main result.} We begin by reviewing the statement of our main result for cocycles. Let $(X,d)$ be a metric space supporting a Borel probability measure $\mu$ and an invertible, measure-preserving transformation $T \colon X \to X$. Fix $m \in \mathbb{N}$ and let $C \colon X \times \mathbb{Z} \to \mathrm{GL}(m,\mathbb{R})$  be a measurable cocycle over $T$. 

Denote by $\langle \cdot, \cdot\rangle$ the standard inner product on $\mathbb{R}^m$, by $\|\cdot\|$ the corresponding Euclidean norm, and by $\nu$ the induced probability measure on the projectivization $\mathbb{P}\mathbb{R}^m$. The projectivized bundle $X \times \mathbb{P}\mathbb{R}^m$ can be endowed with the product measure $\mu \otimes \nu$.

Given $(x,v) \in X \times (\mathbb{R}^m \setminus \{0\})$ or $X \times \mathbb{P}\mathbb{R}^m$ and $N \in \mathbb{N}$ consider the quantities
\begin{align*}
    \sigma(x,v,N) := \log \frac{\|C(x,N) v\|}{\|v\|} \in \mathbb{R},\\
    \sigma(x,N) := \log \sup_{v \in \mathbb{PR}^m} \frac{\|C(x,N)v\|}{\|v\|} \in \mathbb{R}.
\end{align*}
Recall we say the cocycle $C$ is log-integrable if the following maps belong to $L^1(X,\mu)$,
\[
x \mapsto \max\{0,\sigma(x,1)\}, \quad x \mapsto \max\{0,\sigma(x,-1)\}.
\]

Let $\V := (V_N)_{N \in \mathbb{N}}$ be a sequence of positive real numbers with $V_N \to \infty$ as $N \to \infty$. Recall we say the cocycle $C$ is $\V$-sufficiently-bounded if for $\mu$-almost-every $x \in X$, $\nu$-almost every $v \in \mathbb{PR}^m$, and every $N \in \mathbb{N}$, 
\[
|\sigma(x,v,N) - \sigma(Tx,C(x,1)v,N)| = o_{x,v}(V_N) \,.
\]
Recall we say the cocycle $C$ has $\V$-simple-dominated-splitting if for $\mu$-almost every $x \in X$ and $\nu$-almost every $v,w \in \mathbb{PR}^m$, the following holds,
    \[
    |\sigma(x,v,N) - \sigma(x,w,N)| = o_{x,v,w}(V_N).
    \]

Now let $U \colon X \to X$ be a measure-preserving, ergodic transformation and let
$D \colon X \times \mathbb{N} \to \mathrm{GL}(\R^m)$ be a measurable cocycle over $U$.
Recall that we say the pair $(T,U)$ is contracting if for $\mu$-almost-every $x \in X$,
\[
\lim_{n \to \infty} d(T^n U x, T^n x) = 0.
\] 
Recall we say the cocycles $(C,D)$ are $(T,U, \V)$-adapted if for $\mu$-almost-every $x \in X$, $\nu$-almost-every $v \in \mathbb{PR}^m$, and every $N \in \mathbb{N}$, the following estimate holds,
\[
|\sigma(x,v,N) - \sigma(Ux,D(x,1)v,N)| = o_{x,v}(V_N).
\]

The following is the main result of this section; compare to Theorem \ref{theo:cocycle_MCLT_intro}.

\begin{theorem}
\label{theo:cocycle_MCLT}
Let $(X,d)$ be a metric space supporting a Borel probability measure $\mu$, let $T \colon X \to X$ be an invertible, measure-preserving transformation, and let $C \colon X \times \mathbb{Z} \to \mathrm{GL}(m,\mathbb{R})$ be a measurable cocycle over $T$. Assume that $C$ satisfies a DLT on $(X \times \mathbb{R}^m, \mu \otimes \nu)$ with averaging sequence $\A := (A_N)_{N \in \mathbb{N}}$, normalizing sequence $\V=(V_N)_{N \in \mathbb{N}}$, and limiting distribution $S$ in the sense that the random variables
\[
S_N(x,v) := \frac{\sigma(x,v,N) - A_N}{V_N}
\quad \text{ on } \ (X\times \mathbb{P}\R^m, \mu\otimes \nu)
\]
converge in distribution to  $S$ as $N \to \infty$, i.e., for every interval $(a,b) \subseteq \R$ such that $\mathbb{P}(S \in \{a,b\}) = 0$, the following holds,
\[
\lim_{N \to \infty} (\mu \otimes \nu)( \{(x,v) \in X \times \mathbb{PR}^m \ | \ S_N(x) \in (a,b)\})= \mathbb{P}\left(S \in (a,b) \right).
\]
Assume that $C$ is $\V$-sufficiently-bounded and has $\V$-simple-dominated-splitting. Assume in addition that there exist an ergodic transformation $U \colon X \to X$ and a measurable cocycle $D \colon X \times \mathbb{N} \to \mathrm{GL}(m,\mathbb{R})$ over $U$ such that $(T,U)$ is contracting and $(C,D)$ is $(T,U, \V)$-adapted. Then, the above DLT holds in the mixing sense, i.e. for every pair of Borel measurable subsets $A,B \subseteq X$ and every interval $(a,b) \subseteq \R$ such that $\mathbb{P}(S \in \{a,b\}) = 0$,
\begin{gather*}
  \lim_{N \to \infty}(\mu \otimes \nu) (\{ (x,v) \in X \times \mathbb{PR}^m \ | \ x \in A, \ S_N(x,v) \in (a,b), \ T^N x \in B\})\\
  = \mu(A) \cdot \mathbb{P}\left(S \in (a,b) \right) \cdot \mu(B).
\end{gather*}
\end{theorem}

\subsection*{Generic expansion.} Let $(X,\mathcal{B},\mu)$ a probability space supporting an invertible, measure-preserving, ergodic transformation $T \colon X \to X$. Suppose $ C \colon X \times \mathbb{Z} \to \mathrm{GL}(\mathbb{R},m)$ is a measurable cocycle over $T$. Assume $C$ is log-integrable with simple top Lyapunov exponent $\lambda \in \mathbb{R}$. We say a pair $(x,v) \in X \times (\mathbb{R}^m \setminus \{0\})$ is future-Oseledets-generic if
\[
\lim_{N \to \infty} \frac{\sigma(x,v,N)}{N} = \lambda.
\]

The notion of simple dominated splitting is motivated by the following result, which establishes that log-integrable cocycles over ergodic dynamical systems with
simple top Lyapunov exponent have simple dominated
splitting.

\begin{lemma}
    \label{lemma:exp}
    Let $(X,\mathcal{B},\mu)$ be a probability space supporting an invertible, measure-preserving, ergodic transformation $T \colon X \to X$. Suppose $ C \colon X \times \mathbb{Z} \to \mathrm{GL}(\mathbb{R},m)$ is a measurable cocycle over $T$. Assume $C$ is log-integrable with simple top Lyapunov exponent $\lambda \in \mathbb{R}$. Then, for every $x\in X$, every $v,w \in \mathbb{R}^m$ such that the pairs $(x,v),(x,w) \in X \times \mathbb{R}^m$ are future-Oseledets-generic, and every $N \in \mathbb{N}$,
    \[
    |\sigma(x,v,N) - \sigma(x,w,N)| = O_{x,v,w}(1)\,.
    \]
    In particular, $C$ has simple dominated splitting with respect to any diverging sequence $\V$.
    \end{lemma}

\begin{proof}
    Let $X' \subseteq X$ be the full measure subset on which the Oseledets ergodic theorem holds. Fix $x \in X'$ and let $u \in \mathbb{R}^m$ be a vector generating the corresponding top Oseledets subspace. Let $0 <\lambda''< \lambda' < \lambda$ be constants larger than all but the top Lyapunov exponent of $C$. In particular, as $(x,u) \in X \times \mathbb{R}^m$ is future-Oseledets-generic, for every $N \in \mathbb{N}$,
    \[
    \|C(x,N) u\| = \Omega_x(e^{\lambda' N}).
    \]
    
    Now let $(x,v),(x,w) \in X \times \mathbb{R}^m$ be future-Oseledets-generic pairs. Without loss of generality we can assume $v,w \in \mathbb{R}^m$ have unit norm. In particular, for every $N \in \mathbb{N}$,
    \[
    \sigma(x,v,N) = \log \|C(x,N) v\|, \quad \sigma(x,w,N) = \log \|C(x,N) w\|.
    \]
    
    Using the notation above we can write
    \[
    v = au + v', \quad w = bu + w',
    \]
    with $a,b \neq 0$ and $v',w' \in \mathbb{R}^m$ belonging to lower Oseledets subspaces. As the top Lyapunov exponent $\lambda \in \mathbb{R}$ is simple, for every $N \in \mathbb{N}$,
    \begin{align*}
    \|C(x,N) v\| &= |a| \cdot \|C(x,N) u\| + O_v(e^{\lambda'' N}), \\
    \|C(x,N) w\| &= |b| \cdot \|C(x,N) u\| + O_w(e^{\lambda'' N}).
    \end{align*}
    In particular, for every $N \in \mathbb{N}$,
    \begin{align*}
    \sigma(x,v,N) = \log |a| + \log \|C(x,N) u\| + \log(1 + o_v(1)),\\
    \sigma(x,w,N) = \log |b| + \log \|C(x,N) u\| + \log(1 + o_w(1)).
    \end{align*}
    The desired estimate follows.
\end{proof}

\subsection*{A conditional distributional limit theorem.} To prove Theorem \ref{theo:cocycle_MCLT} we first prove the following result of independent interest.

\begin{theorem}
\label{theo:cocycle2_CCLT}
Let $(X,\mathcal{B},\mu)$ be a probability space, let $T \colon X \to X$ be an invertible, measure-preserving, ergodic  transformation, and let $C \colon X \times \mathbb{Z} \to \mathrm{GL}(m,\mathbb{R})$ be a measurable cocycle over $T$. Assume that $C$ satisfies a DLT on $(X \times \mathbb{PR}^m, \mu \otimes \nu)$ with averaging sequence $\A := (A_N)_{N \in \mathbb{N}}$, normalizing sequence $\V=(V_N)_{N \in \mathbb{N}}$, and limiting distribution $S$ in the sense that the random variables
\[
S_N(x,v) := \frac{\sigma(x,v,N) -A_N}{V_N}
\ \ \text{ on }  (X\times \mathbb{PR}^m, \mu\otimes \nu)
\]
converge in distribution to $S$ as $N \to \infty$, i.e., for every interval $(a,b) \subseteq \R$ such that $\mathbb{P}(S \in \{a,b\}) = 0$, the following holds,
\[
\lim_{N \to \infty} (\mu \otimes \nu)( \{(x,v) \in X \times \mathbb{PR}^m \ | \ S_N(x) \in (a,b)\}) = \mathbb{P}\left(S \in (a,b) \right).
\]
Assume the cocycle $C$ is $\V$-sufficiently-bounded and has $\V$-simple-dominated-splitting. Then, the above DLT holds in the conditional sense, i.e. for every $A \in \mathcal{B}$ and every interval $(a,b) \subseteq \R$ such that $\mathbb{P}(S \in \{a,b\}) = 0$,
\begin{gather*}
  \lim_{N \to \infty} (\mu \otimes \nu)( \{(x,v) \in X \times \mathbb{PR}^m \ | \ x \in A, \ \ S_N(x) \in (a,b)\})
  = \mu(A) \cdot \mathbb{P}\left(S \in (a,b) \right).
\end{gather*}
\end{theorem}

Before proving Theorem \ref{theo:cocycle2_CCLT}, we first introduce some terminology. Let $(\Omega,\mathcal{B},\mathbb{P})$ be a probability space, $\mathcal{F} \subseteq \mathcal{B}$ be a sub-$\sigma$-algebra, and $(X_n)_{n \in \N}$ be a sequence of integrable random variables on $(\Omega,\mathcal{B},\mathbb{P})$. We say that $X_n$ converges weakly in $L^1$ relative to $\mathcal{F}$ to an integrable random variable $X$ on $(\Omega,\mathcal{B},\mathbb{P})$ if for every bounded random variable $Y$ on $(\Omega,\mathcal{B},\mathbb{P})$  that is $\mathcal{F}$-measurable the following holds,
\[
\lim_{n \to \infty} \mathbb{E}(X_n Y) = \mathbb{E}(XY).
\]
Equivalently, $X_n$ converges weakly in $L^1$ relative to $\mathcal{F}$ to $X$ if for every $A \in \mathcal{F}$,
\[
\lim_{n \to \infty} \mathbb{E}(X_n \chi_A) = \mathbb{E}(X\chi_A).
\]

To disregard regularity considerations and to reduce to a question about relative weak $L^1$ convergence, we prove the following general lemma; compare to Lemma \ref{lemma:tech}.

\begin{lemma}
    \label{lemma:tech2}
    Let $(\Omega,\mathcal{B},\mathbbm{P})$ be a probability space, $\mathcal{F} \subseteq \mathcal{B}$ be a sub-$\sigma$-algebra, and $(X_n)_{n \in \N}$ be a sequence of integrable random variables on $(\Omega,\mathcal{B},\mathbbm{P})$. Then $X_n$ converges conditionally in distribution relative to $\mathcal{F}$ to an integrable random variable $X$ in the sense that for every $A \in \mathcal{F}$ and every interval $(a,b) \subseteq \mathbb{R}$ such that $\mathbb{P}(X \in \{a,b\}) = 0$,
    \[
    \lim_{n \to \infty} \mathbb{P}(A \cap X_n \in (a,b)) =  \mathbb{P}(A) \cdot \mathbb{P}(X \in (a,b)), 
    \]
    if and only for every fixed $t \in \R$,
    \[
    e^{itX_n} \to \mathbb{E}(e^{itX}) \ \ \text{weakly in }L^1 \text{ relative to }\mathcal{F} \text{ as }n \to \infty.
    \]
\end{lemma}

\begin{proof}
    For every $A \in \mathcal{F}$ with $\mathbb{P}(A) > 0$ consider the probability space $(A,\mathcal{B}|_A,\mathbb{P}|_A)$ and the integrable random variables $(X_n|_A)_{n \in \N}$ on it. Notice that $X_n$ converges conditionally in distribution relative to $\mathcal{F}$ to $X$ if and only if $(X_n|_A)_{n \in \N}$ converges in distribution to $X$ for every $A \in \mathcal{F}$ with $\mathbb{P}(A) > 0$. By L\'evy's continuity theorem, this condition is equivalent to the following convergence of characteristic functions for every $t \in \R$ :
    \[
    \mathbb{E}(e^{it X_n}|A) \to \mathbb{E}(e^{it X}) \ \ \text{as } n \to \infty.
    \]
    But this is exactly the condition that for every set $A \in \mathcal{F}$ with $\mathbb{P}(A) > 0$, 
    \[
    \mathbb{E}(\chi_A e^{itX_n}) \to \mathbb{E}(\chi_A\mathbb{E}(e^{itX})).
    \]
    Thus, we see that the original condition is equivalent to
    \[
    e^{itX_n} \to \mathbb{E}(e^{itX}) \ \ \text{weakly in }L^1 \text{ relative to } \mathcal{F} \text{ for every }t\in \R. \qedhere
    \]
\end{proof}

We are now ready to prove Theorem \ref{theo:cocycle2_CCLT}

\begin{proof}[Proof of Theorem \ref{theo:cocycle2_CCLT}]
    Denote by $\mathcal{B}(\mathbb{PR}^m)$ the Borel $\sigma$-algebra of $\mathbb{PR}^m$. Consider the normalized expansion rates $S_N$ as random variables on $(X \times \mathbb{PR}^m,\mathcal{B} \otimes \mathcal{B}(\mathbb{PR}^m),\mu \otimes \nu)$. Denote by $\mathcal{F} \subseteq \mathcal{B} \otimes \mathcal{B}(\mathbb{PR}^m)$ the sub-$\sigma$-algebra induced by the projection $X \times \mathbb{PR}^m \to X$. By Lemma \ref{lemma:tech2}, it is enough for our purposes to check that  for every fixed $t \in \R$,
    \[
    e^{itS_N} \to \mathbb{E}(e^{itS}) \ \ \text{weakly in }L^1 \text{ relative to } \mathcal{F} \text{ as } N \to \infty.
    \]
    
    For the rest of this discussion we fix $t \in \R$ and prove this statement. Notice that this statement is equivalent to showing that the random variables $Y_N$ on $(X,\mathcal{B},\Omega)$ given by
    \[
    Y_N(x) := \int_{\mathbb{PR}^m} e^{itS_N(x,v)} \thinspace d\nu(v), \ \ N \in \mathbb{N},
    \]
    converge weakly in $L^1$ to $\mathbb{E}(e^{itS})$ as $N \to \infty$. By the Dunford-Pettis theorem, a sequence of integrable random variables is sequentially $L^1$ weakly compact if and only if it is uniformly integrable. In particular, as the random variables $Y_N$ are uniformly bounded, it is enough to show that the only weak $L^1$ limit point of $(Y_N)_{N \in \N}$ is the constant random variable $\mathbb{E}(e^{itS})$.
    
    Let $Y$ be a weak $L^1$ limit point of this sequence along times $\{N_k\}_{k \in \N}$. We aim to show that $Y = Y \circ T$ almost surely. To prove this we first verify the claim:
    \begin{equation}
    \label{eq:AAA}
    \mathbb{E}(|Y_N \circ T - Y_N|) \to 0 \ \ \text{as }N \to \infty.
    \end{equation}
    Indeed, consider the sequence of random variables $(Z_N)_{N \in \mathbb{N}}$ on $(X,\mathcal{B},\Omega)$ given by
    \[
    Z_N(x) := \int_{\mathbb{PR}^m} e^{itS_N(x,C(T^{-1}x,1)v)} d\nu(v).
    \]
    Then, for every $N \in \mathbb{N}$ we can write
    \begin{gather*}
    \mathbb{E}(|Y_N \circ T - Y_N|) \leq \mathbb{E}(|Y_N \circ T - Z_N \circ T|) + \mathbb{E}(|Z_N \circ T - Y_N|).
    \end{gather*}
    The first term converges to zero as $N \to \infty$ because of the assumption that $C$ has $\V$-simple-dominated-splitting and the dominated convergence theorem. The second term converges to zero as $N \to \infty$ because of the assumption that $C$ is $\V$-sufficiently-bounded and the dominated convergence theorem. This completes the proof of the claim.
    
    The claim together with the Hölder inequality implies that
    \[
    Y_{N_k} \circ T \to Y \ \ \text{weakly in }L^1 \text{ as }k \to \infty.
    \]
    At the same time, the $T$-invariance of $\mu$ implies that
    \[
    Y_{N_k} \circ T \to Y \circ T \ \ \text{weakly in }L^1 \text{ as } k \to \infty.
    \]
    As weak $L^1$ limits are unique almost surely, this shows that $Y  = Y \circ T$ almost surely.
    
    To conclude we notice that, as $T$ acts ergodically, $Y$ must be constant almost everywhere. Pairing $Y_{N_k}$ against the constant function $1$ and using the assumed DLT  reveals that $Y$ must be $\mathbb{E}(e^{itS})$. This finishes the proof.
\end{proof}

The proof of Theorem \ref{theo:cocycle_MCLT} will use Theorem \ref{theo:cocycle2_CCLT} in the following form.

\begin{corollary}
\label{cor:tech_cocycle}
Let $(X,\mathcal{B},\mu)$ be a probability space, let $T \colon X \to X$ be an invertible, measure-preserving, ergodic  transformation, and let $C \colon X \times \mathbb{Z} \to \mathrm{GL}(m,\mathbb{R})$ be a measurable cocycle over $T$. Assume that $C$ satisfies a DLT on $(X \times \mathbb{PR}^m, \mu \otimes \nu)$ with averaging sequence $\A := (A_N)_{N \in \mathbb{N}}$, normalizing sequence $\V=(V_N)$, and limiting distribution $S$ in the sense that the random variables
\[
S_N(x,v) := \frac{\sigma(x,v,N) -A_N}{V_N}
\ \ \text{ on }  (X\times \mathbb{PR}^m, \mu\otimes \nu)
\]
converge in distribution to $S$ as $N \to \infty$, i.e., for every interval $(a,b) \subseteq \R$ such that $\mathbb{P}(S \in \{a,b\}) = 0$, the following holds,
\[
\lim_{N \to \infty} (\mu \otimes \nu)( \{(x,v) \in X \times \mathbb{PR}^m \ | \ S_N(x) \in (a,b)\}) = \mathbb{P}\left(S \in (a,b) \right).
\]
Assume the cocycle $C$ is $\V$-sufficiently-bounded and has $\V$-simple-dominated-splitting. Then, the same DLT holds in the following stronger sense: for every $B \in \mathcal{B}$ and every interval $(a,b) \subseteq \R$ such that $\mathbb{P}(S \in \{a,b\}) = 0$,
\begin{gather*}
  \lim_{N \to \infty} (\mu \otimes \nu)( \{(x,v) \in X \times \mathbb{PR}^m \ | \ S_N(x) \in (a,b), \ T^N x \in B\}) \\
  = \mathbb{P}\left(S \in (a,b) \right) \cdot \mu(B).
\end{gather*}
Furthermore, if $X$ is a metric space and $\mathcal{B}$ is its Borel $\sigma$-algebra, then for every bounded Lipschitz function $\phi \colon X \to \R$ and every $t \in \R$,
\begin{gather*}
  \lim_{N \to \infty}  \int_{X} \int_{\mathbb{PR}^m} e^{itS_N(x,v)} \cdot \phi(T^N x)  \thinspace d\nu(v) \thinspace d\mu(x) 
  = \mathbb{E}(e^{itS}) \cdot \mu(\phi).
\end{gather*}
\end{corollary}


\begin{proof}
    Without loss of generality consider $B \in \mathcal{B}$ with $\mu(B) > 0$. The $T$-invariance of $\mu$ ensures that, for every $N \in \N$,
    \[
    \mu(T^{-N} B) = \mu(B) > 0.
    \]
    For every $N \in \N$ consider the probability space $(T^{-N} B,\mathcal{B}|_{T^{-N}B},\mu|_{T^{-N} B} )$ and the random variable $S_N|_{T^{-N} B}$ on it. Our goal is to show that $(S_N|_{T^{-N}B})_{N \in \N}$ converges in distribution to $S$. By L\'evy's continuity theorem, this condition is equivalent to the following convergence of characteristic functions for every $t \in \R$:
    \[
    \mathbb{E}(e^{itS_N} | T^{-N}B) \to \mathbb{E}(e^{itS}) \ \ \text{as }N \to \infty.
    \]
    For the rest of this discussion we fix $t \in \R$ and show that
    \[
    \lim_{N \to \infty} \int_X  \int_{\mathbb{PR}^m} e^{itS_N(x,v)} \cdot \chi_B(T^Nx) \thinspace d\nu(v) \thinspace d\mu(x) = \mathbb{E}(e^{itS}) \cdot \mu(B).
    \]

    Notice that, as $T$ is measure-preserving, 
    \begin{gather*}
        \int_X  \int_{\mathbb{PR}^m} e^{itS_N(x,v)} \cdot \chi_B(T^Nx) \thinspace d\nu(v) \thinspace d\mu(x)\\
       = \int_X  \int_{\mathbb{PR}^m} e^{itS_N(T^{-N}x,v)} \cdot \chi_B(x) \thinspace d\nu(v) \thinspace d\mu(x).
    \end{gather*}
    For every $N \in \mathbb{N}$ consider the random variable $Y_N'$ on $(X,\mathcal{B},\Omega)$ given by
    \[
    Y_N'(x) := \int_{\mathbb{PR}^m} e^{itS_N(T^{-N}x,v)} \thinspace d\nu(v).
    \]
    Our goal is to show that the variables $Y_N'$ converge weakly in $L^1$ to $\mathbb{E}(e^{itS})$ as $N \to \infty$. By the Dunford-Pettis theorem, a sequence of integrable random variables is sequentially $L^1$ weakly compact if and only if it is uniformly integrable. In particular, as the random variables $Y_N'$ are uniformly bounded, it is enough to show that the only weak $L^1$ limit point of $(Y_N')_{N \in \N}$ is the constant random variable $\mathbb{E}(e^{itS})$.
    
    Let $Y'$ be a weak $L^1$ limit point of this sequence along times $\{N_k\}_{k \in \N}$. We aim to show that $Y' = Y' \circ T$ almost surely. To prove this we first verify the claim:
    \[
    \mathbb{E}(|Y_N' \circ T - Y_N'|) \to 0 \ \ \text{as }N \to \infty.
    \]
    Notice that, as $T$ is measure-preserving, for every $N \in \mathbb{N}$,
    \[
    \mathbb{E}(|Y_N' \circ T - Y_N'|) = \mathbb{E}(|Y_N'\circ T^N \circ T - Y_N' \circ T^N|).
    \]
    Directly from the definitions one can check that $Y_N' \circ T^N  = Y_N$ for every $N \in \mathbb{N}$, with $Y_N$ as defined in the proof of Theorem \ref{theo:cocycle2_CCLT}. The proof of the claim follows from \eqref{eq:AAA}.
    
    The claim together with the Hölder inequality implies that
    \[
    Y_{N_k}' \circ T \to Y' \ \ \text{weakly in }L^1 \text{ as }k \to \infty.
    \]
   At the same time, the $T$-invariance of $\mu$ implies
    \[
    Y_{N_k}' \circ T \to Y' \circ T \ \ \text{weakly in }L^1 \text{ as } k \to \infty.
    \]
    As weak $L^1$ limits are unique almost surely, this shows that $Y'  = Y' \circ T$ almost surely.
    
    To conclude notice that, as $T$ acts ergodically, $Y'$ must be constant almost everywhere. Pairing $Y_{N_k}'$ against the constant function $1$, using the assumed DLT and the $T$-invariance of $\mu$ reveals that $Y'$ must be $\mathbb{E}(e^{itS})$. This finishes the proof.
 \end{proof}

\subsection*{Proof of the main result.} We are now ready to prove Theorem \ref{theo:cocycle_MCLT}.

\begin{proof}[Proof of Theorem \ref{theo:cocycle_MCLT}]
    Without loss of generality consider a pair of Borel measurable sets $A,B \subseteq X$ with $\mu(A),\mu(B) > 0$. Theorem \ref{theo:mix} ensures $T$ is mixing. In particular,
    \[
    \lim_{N \to \infty} \mu(A \cap T^{-N} B)  = \mu(A) \cdot \mu(B) > 0.
    \]
    For every $N \in \N$ sufficiently large consider the probability space $$(A \cap T^{-N} B,\mathcal{B}|_{A \cap T^{-N}B},\mu|_{A \cap T^{-N} B} )$$ and the random variable $S_N|_{A \cap T^{-N} B}$ on it. Our goal is to show that $(S_N|_{A \cap T^{-N}B})_{N \in \N}$ converges in distribution to $S$. By L\'evy's continuity theorem, this condition is equivalent to the following convergence of characteristic functions for every $t \in \R$ :
    \[
    \mathbb{E}(e^{itS_N} | A \cap T^{-N}B) \to \mathbb{E}(e^{itS}) \ \ \text{as }N \to \infty.
    \]
    As $T$ is mixing, this is equivalent to showing that, for every $t \in \R$,
    \[
    \lim_{N \to \infty} \int_X \int_{\mathbb{PR}^m} \chi_A(x) \cdot e^{itS_N(x,v)} \cdot \chi_B(T^Nx) \thinspace d\nu(v) \thinspace d\mu(x) = \mu(A) \cdot \mathbb{E}(e^{itS}) \cdot \mu(B).
    \]
    A standard approximation argument shows that we can replace $\chi_B$ in this statement by a bounded Lipschitz function $\phi \colon X \to \R$. We do so for the rest of this discussion. Now consider for fixed $t \in \R$ the random variables $F_N$ on $(X,\mathcal{B}(X),\mu)$ given by
    \[
    F_N(x) := \left( \int_{\mathbb{PR}^m} e^{itS_N(x,v)} \thinspace d\nu(v) \right) \cdot \phi(T^N x), \ \ N \in \N.
    \]
    Our goal is to show that 
    \[
    F_N \to \mathbb{E}(e^{itS}) \cdot \mu(\phi) \ \ \text{weakly in }L^1 \text{ as } N \to \infty.
    \]
    
    By the Dunford-Pettis theorem, a sequence of integrable random variables is sequentially $L^1$ weakly compact if and only if it is uniformly integrable. In particular, as the random variables $F_N$ are uniformly bounded, it is enough to show that the only weak $L^1$ limit point of $(F_N)_{N \in \N}$ is the constant random variable $\mathbb{E}(e^{itS}) \cdot \mu(\phi)$.
    
    Let $F$ be a weak $L^1$ limit point of this sequence along times $\{N_k\}_{k \in \N}$. We aim to show that $F = F \circ U$ almost surely. To prove this we first verify the claim:
    \begin{equation*}
    \mathbb{E}(|F_N \circ U - F_N|) \to 0 \ \ \text{as }N \to \infty.
    \end{equation*}
    Indeed, consider the sequence of random variables $(Z_N)_{N \in \mathbb{N}}$ on $(X,\mathcal{B}(X),\mu)$ given by
    \[
    Z_N(x) := \left(\int_{\mathbb{PR}^m} e^{itS_N(x,D(U^{-1}x,1)v)} d\nu(v)\right) \cdot \phi(T^Nx) .
    \]
    Then, for every $N \in \mathbb{N}$ we can write
    \begin{gather*}
    \mathbb{E}(|F_N \circ U - F_N|) \leq \mathbb{E}(|F_N \circ U - Z_N \circ U|) + \mathbb{E}(|Z_N \circ U - F_N|).
    \end{gather*}
    The first term converges to zero as $N \to \infty$ because of the assumption that $C$ has $\V$-simple-dominated-splitting, the fact that $\phi$ is bounded, and the dominated convergence theorem. The second term converges to zero as $N \to \infty$ because of the assumption that $(T,U)$ is contracting, the assumption that $(C,D)$ is $(T,U, \V)$-adapted, the fact that $\phi$ is bounded Lipschitz, and the dominated convergence theorem. The claim is thus proved.
    
    The claim together with the Hölder inequality implies that
    \[
    F_{N_k} \circ U \to F \ \ \text{weakly in }L^1 \text{ as }k \to \infty.
    \]
    At the same time, the $U$-invariance of $\mu$ implies
    \[
    F_{N_k} \circ U \to F \circ U \ \ \text{weakly in }L^1 \text{ as } k \to \infty.
    \]
    As weak $L^1$ limits are unique almost surely, this shows that $F  = F \circ U$ almost surely.
    
    To conclude we notice that, as $U$ acts ergodically, $F$ must be constant almost everywhere. Pairing $F_{N_k}$ against the constant function $1$ and using Corollary \ref{cor:tech_cocycle} reveals that $F$ must be $\mathbb{E}(e^{itS}) \cdot \mu(\phi)$. This finishes the proof.
\end{proof}



\subsection*{The operator norm.} We now discuss how to deduce (mixing) distributional limit theorems for the operator norm of cocyles over dynamical systems from the (mixing) distributional limit theorems studied above. To this end we first introduce a strengthened version of the 
property of simple dominated splitting. 

Let $(X,\mathcal{B},\mu)$ a probability space, let $T \colon X \to X$ an invertible, measure-preserving transformation, and let $ C \colon X \times \mathbb{Z} \to \mathrm{GL}(\mathbb{R},m)$ a measurable cocycle over $T$. Consider a sequence of positive real numbers $\V := (V_N)_{N \in \mathbb{N}}$ with $V_N \to \infty$ as $N \to \infty$. We say the cocycle $C$ has $\V$-strong-simple-dominated-splitting  if for $\mu$-almost-every $x \in X$, $\nu$-almost-every $v \in \mathbb{PR}^m$, and every $N \in \mathbb{N}$, the following estimate holds: 
    \[
    |\sigma(x,v,N) - \sigma(x,N)| = o_{x,v}( V_N).
    \]

In applications, the main tool we will use is the following strengthening of Lemma \ref{lemma:exp}.

\begin{lemma}
    \label{lemma:exp2}
    Let $(X,\mathcal{B},\mu)$ be a probability space supporting an invertible, measure-preserving, ergodic transformation $T \colon X \to X$. Suppose $ C \colon X \times \mathbb{Z} \to \mathrm{GL}(\mathbb{R},m)$ is a measurable cocycle over $T$. Assume $C$ is log-integrable with simple top Lyapunov exponent $\lambda \in \mathbb{R}$. Then, for every $x \in X$,every $v \in \mathbb{PR}^m$ such that the pair $(x,v) \in X \times \mathbb{R}^m$ is future-Oseledets-generic, and every $N \in \mathbb{N}$, the following estimate holds,
    \[
    |\sigma(x,v,N) - \sigma(x,N)| = O_{x,v}(1).
    \]
    In particular, $C$ has a strong dominated splitting with respect to any diverging sequence.
\end{lemma}

\begin{proof}
    Let $X' \subseteq X$ be the full measure subset on which the Oseledets ergodic theorem holds. Fix $x \in X'$ and let $u \in \mathbb{R}^m$ be the unit norm vector generating the corresponding top Oseledets subspace. Let $0 <\lambda''< \lambda' < \lambda$ be constants larger than all but the top Lyapunov exponent of $C$. In particular, as $(x,u) \in X \times \mathbb{R}^m$ is future-Oseledets-generic, the following estimate holds for every $N \in \mathbb{N}$ ,
    \[
    \|C(x,N) u\| = \Omega_x(e^{\lambda' N}).
    \]
    
    Now let $(x,v) \in X \times \mathbb{R}^m$  be arbitrary. Without loss of generality we can assume $v \in \mathbb{R}^m$ has unit norm. In particular, for every $N \in \mathbb{N}$,
    \[
    \sigma(x,u,N) = \log \|C(x,N) u\|, \quad \sigma(x,v,N) = \log \|C(x,N) v\|.
    \]
    
    Using the notation above we can write
    \[
    v = au + v', 
    \]
    with $a \in \R$ and $v' \in \mathbb{R}^m$ belonging to lower Oseledets subspaces. Since
    $v$  has unit norm, 
    $$
    a =O_x (1)\,.
    $$
     As the top Lyapunov exponent $\lambda \in \mathbb{R}$ is simple, for every $N \in \mathbb{N}$,
      \[
    \|C(x,N) v\| = |a| \cdot \|C(x,N) u\| + O_x(e^{\lambda'' N}).
     \]
    In particular, there exists $N_{x,v} \in \N$ 
    and $a_x \in \R$ such that for every $N\geq N_{x,v}$, 
    \begin{align*}
    \sigma(x,v,N) &= \log |a| + \log \|C(x,N) u\| + \log(1 + o_x(1)),\\
    \sigma(x,N) &= \log |a_x| + \log \|C(x,N) u\| + \log(1 + o_x(1)) .
    \end{align*}
    The desired estimate follows.
\end{proof}

We now discuss distributional limit theorems with respect to the operator norm.

\begin{theorem}
\label{theo:cocycle_CCLT_norm}
Let $(X,\mathcal{B},\mu)$ be a probability space, let $T \colon X \to X$ be an invertible, measure-preserving, ergodic transformation, and let
$ C \colon X \times \mathbb{Z} \to \mathrm{GL}(\mathbb{R},m)$ be a measurable cocycle over $T$. Assume that $C$ satisfies a DLT on $(X \times \mathbb{PR}^m, \mu \otimes \nu)$ with averaging sequence $\A := (A_N)_{N \in \mathbb{N}}$, normalizing sequence $\V :=(V_N)_{N \in \mathbb{N}}$, and limiting distribution $S$ in the sense that the random variables
\[
S_N(x,v) := \frac{\sigma(x,v,N) - A_N}{V_N} \ \text{ on }  (X\times \mathbb{P}\R^m, \mu \otimes \nu)
\]
converge in distribution to  $S$ as $N \to \infty$, i.e., for every interval $(a,b) \subseteq \R$ such that $\mathbb{P}(S \in \{a,b\}) = 0$, the following holds,
\[
\lim_{N \to \infty} (\mu \otimes \nu)(S_N(x,v) \in (a,b)) = \mathbb{P}\left(S \in (a,b) \right).
\]
Assume $C$ is has $\V$-strong-simple-dominated-splitting. Then, $C$ satisfies the above DLT with respect to the operator norm in the sense that the random variables
\[
S_N(x) := \frac{\sigma(x,N)-A_N}{V_N} \  \text{ on }  (X, \mu)
\]
converge in distribution to $S$ as $N \to \infty$, i.e., for every interval $(a,b) \subseteq \R$ such that $\mathbb{P}(S \in \{a,b\}) = 0$, the following holds,
\begin{gather*}
  \lim_{N \to \infty} \mu(S_N(x) \in (a,b))
  = \mathbb{P}\left(S \in (a,b) \right).
\end{gather*}
\end{theorem}

\begin{proof}
    By L\'evy's continuity theorem, the first DLT is equivalent to the following convergence of characteristic functions for every $t \in \R$:
    \[
    \int_X \int_{\mathbb{PR}^m} e^{itS_N(x,v)} \thinspace d\nu(v) \thinspace d\mu(x) \to \mathbb{E}(e^{itS}) \ \ \text{as }N \to \infty.
    \]
    Analogously, the second DLT is equivalent to the following convergence of characteristic functions for every $t \in \R$:
    \[
    \int_X e^{itS_N(x)} \thinspace d\mu(x) \to \mathbb{E}(e^{itS}) \ \ \text{as }N \to \infty.
    \]
    It is enough then to show that 
    \[
     \left|\int_X \int_{\mathbb{PR}^m} e^{itS_N(x,v)} \thinspace d\nu(v) \thinspace d\mu(x) - \int_X e^{itS_N(x)} \thinspace d\mu(x)\right| \to 0 \ \ \text{as } N \to \infty.
    \]
    This follows from the assumption that $C$ has a $\V$-strong-simple-dominated-splitting
    and the dominated convergence theorem.
\end{proof}

We end this discussion with the following result whose proof is analogous to that of Theorem \ref{theo:cocycle_CCLT_norm}; details are left to the reader.

\begin{theorem}
\label{theo:cocycle_CCLT_norm2}
Let $(X,\mathcal{B},\mu)$ be a probability space, let $T \colon X \to X$ be an invertible, measure-preserving, ergodic transformation, and let
$ C \colon X \times \mathbb{Z} \to \mathrm{GL}(\mathbb{R},m)$ be a measurable cocycle over $T$. Assume that $C$ satisfies a mixing DLT on $(X \times \mathbb{PR}^m, \mu \otimes \nu)$ with averaging sequence $\A := (A_N)_{N \in \mathbb{N}}$, normalizing sequence $\V :=(V_N)_{N \in \mathbb{N}}$, and limiting distribution $S$ in the sense that for the random variables
\[
S_N(x,v) := \frac{\sigma(x,v,N) - A_N}{V_N} \
\text{ on } (X\times \mathbb{P}\R^m, \mu \otimes \nu)\,,
\]
for every pair of measurable subsets $A,B \in \mathcal{B}$, and for every interval $(a,b) \subseteq \R$ such that $\mathbb{P}(S \in \{a,b\}) = 0$,
the following holds,
\begin{gather*}
  \lim_{N \to \infty}(\mu \otimes \nu) (\{ (x,v) \in X \times \mathbb{PR}^m \ | \ x \in A, \ S_N(x,v) \in (a,b), \ T^N x \in B\})\\
  = \mu(A) \cdot \mathbb{P}\left(S \in (a,b) \right) \cdot \mu(B).
\end{gather*}
Assume that $C$ has $\V$-strong-simple-dominated-splitting.
Then, $C$ satisfies the above mixing DLT  with respect to the operator norm in the sense that 
for the random variables
\[
S_N(x) := \frac{\sigma(x,N)-A_N}{V_N} \ \text{ on } (X, \mu),
\]
for every pair of measurable subsets $A,B \in \mathcal{B}$, and for every interval $(a,b) \subseteq \R$ such that $\mathbb{P}(S \in \{a,b\}) = 0$,
the following holds,
\begin{gather*}
  \lim_{N \to \infty} \mu(\{ x \in X  \ | \ x \in A, \ S_N(x) \in (a,b), \ T^N x \in B\})
  = \mu(A) \cdot \mathbb{P}\left(S \in (a,b) \right) \cdot \mu(B).
\end{gather*}
\end{theorem}

\subsection*{Generic sections.} We now briefly discuss how to deduce (mixing) distributional limit theorems for so-called generic sections of cocyles over dynamical systems from the (mixing) distributional limit theorems studied above. 

Let $(X,\mathcal{B},\mu)$ be a probability space, $T \colon X \to X$ be an invertible, measure-preserving transformation, and $ C \colon X \times \mathbb{Z} \to \mathrm{GL}(\mathbb{R},m)$ be a measurable cocycle over $T$. Consider a positive sequence $\V := (V_N)_{N \in \mathbb{N}}$ with $V_N \to \infty$ as $N \to \infty$. We say a measurable section $s \colon X \to \mathbb{R}^m$ is $(C,\mathcal{V})$-generic if for $\mu$-almost-every $x \in X$, $\nu$-almost every $w \in \mathbb{PR}^m$, and every $N \in \mathbb{N}$, the following estimate holds,
    \[
    |\sigma(x,v(x),N) - \sigma(x,w,N)| = o_{x,w}( V_N).
    \]
    
In applications, the main tool we use is the following result; compare to Lemmas \ref{lemma:exp} and \ref{lemma:exp2}. The proof is omitted but follows from the same arguments as the cited lemmas.

\begin{lemma}
    \label{lemma:exp3}
   Let $(X,\mathcal{B},\mu)$ be a probability space supporting an invertible, measure-preserving, ergodic transformation $T \colon X \to X$. Suppose $ C \colon X \times \mathbb{Z} \to \mathrm{GL}(\mathbb{R},m)$ is a measurable cocycle over $T$. Assume $C$ is log-integrable with simple top Lyapunov exponent $\lambda \in \mathbb{R}$. Let $s \colon X \to \mathbb{R}^m$ be a measurable section such that $(x,v(x)) \in X \times \mathbb{R}^m$ is future-Oseledets-generic for $\mu$-almost every $x \in X$. Then, for every $x \in X$, every $w \in \mathbb{PR}^m$ such that $(x,w) \in X \times \mathbb{R}^m$ is future-Oseledets-generic, and every $N \in \mathbb{N}$,
    \[
    |\sigma(x,v(x),N) - \sigma(x,w,N)| = O_{x}(1).
    \]
    In particular, the section $s \colon X \to \mathbb{R}^m$ is $(C,\mathcal{V})$-generic for any diverging sequence $\mathcal{V}$.
\end{lemma}

We now state two limit theorems for generic sections; compare to Theorems \ref{theo:cocycle_CCLT_norm} and \ref{theo:cocycle_CCLT_norm2}. The proofs are omitted but follow from the same arguments as the cited theorems.

\begin{theorem}
\label{theo:cocycle_CCLT_sec}
Let $(X,\mathcal{B},\mu)$ be a probability space, let $T \colon X \to X$ be an invertible, measure-preserving, ergodic transformation, and let
$ C \colon X \times \mathbb{Z} \to \mathrm{GL}(\mathbb{R},m)$ be a measurable cocycle over $T$. Assume that $C$ satisfies a DLT on $(X \times \mathbb{PR}^m, \mu \otimes \nu)$ with averaging sequence $\A := (A_N)_{N \in \mathbb{N}}$, normalizing sequence $\V :=(V_N)_{N \in \mathbb{N}}$, and limiting distribution $S$ in the sense that the random variables
\[
S_N(x,v) := \frac{\sigma(x,v,N) - A_N}{V_N} \ \text{ on }  (X\times \mathbb{P}\R^m, \mu \otimes \nu)
\]
converge in distribution to  $S$ as $N \to \infty$, i.e., for every interval $(a,b) \subseteq \R$ such that $\mathbb{P}(S \in \{a,b\}) = 0$, the following holds,
\[
\lim_{N \to \infty} (\mu \otimes \nu)(S_N(x,v) \in (a,b)) = \mathbb{P}\left(S \in (a,b) \right).
\]
Let $s \colon X \to \mathbb{R}^m$ be a measurable, $(C,\mathcal{V})$-generic section. Then, $s$ satisfies the above DLT in the sense that the random variables
\[
S_N(x) := \frac{\sigma(x,v(x),N)-A_N}{V_N} \  \text{ on }  (X, \mu)
\]
converge in distribution to $S$ as $N \to \infty$, i.e., for every interval $(a,b) \subseteq \R$ such that $\mathbb{P}(S \in \{a,b\}) = 0$, the following holds,
\begin{gather*}
  \lim_{N \to \infty} \mu(S_N(x) \in (a,b))
  = \mathbb{P}\left(S \in (a,b) \right).
\end{gather*}
\end{theorem}

\begin{theorem}
\label{theo:cocycle_CCLT_sec2}
Let $(X,\mathcal{B},\mu)$ be a probability space, let $T \colon X \to X$ be an invertible, measure-preserving, ergodic transformation, and let
$ C \colon X \times \mathbb{Z} \to \mathrm{GL}(\mathbb{R},m)$ be a measurable cocycle over $T$. Assume that $C$ satisfies a mixing DLT on $(X \times \mathbb{PR}^m, \mu \otimes \nu)$ with averaging sequence $\A := (A_N)_{N \in \mathbb{N}}$, normalizing sequence $\V :=(V_N)_{N \in \mathbb{N}}$, and limiting distribution $S$ in the sense that for the random variables
\[
S_N(x,v) := \frac{\sigma(x,v,N) - A_N}{V_N} \
\text{ on } (X\times \mathbb{P}\R^m, \mu \otimes \nu)\,,
\]
for every pair of measurable subsets $A,B \in \mathcal{B}$, and for every interval $(a,b) \subseteq \R$ such that $\mathbb{P}(S \in \{a,b\}) = 0$,
the following holds,
\begin{gather*}
  \lim_{N \to \infty}(\mu \otimes \nu) (\{ (x,v) \in X \times \mathbb{PR}^m \ | \ x \in A, \ S_N(x,v) \in (a,b), \ T^N x \in B\})\\
  = \mu(A) \cdot \mathbb{P}\left(S \in (a,b) \right) \cdot \mu(B).
\end{gather*}
Let $s \colon X \to \mathbb{R}^m$ be a measurable, $(C,\mathcal{V})$-generic section. Then, $s$ satisfies the above mixing DLT in the sense that the random variables
\[
S_N(x) := \frac{\sigma(x,s(x),N)-A_N}{V_N} \ \text{ on } (X, \mu),
\]
for every pair of measurable subsets $A,B \in \mathcal{B}$, and for every interval $(a,b) \subseteq \R$ such that $\mathbb{P}(S \in \{a,b\}) = 0$,
the following holds,
\begin{gather*}
  \lim_{N \to \infty} \mu(\{ x \in X  \ | \ x \in A, \ S_N(x) \in (a,b), \ T^N x \in B\})
  = \mu(A) \cdot \mathbb{P}\left(S \in (a,b) \right) \cdot \mu(B).
\end{gather*}
\end{theorem}

\subsection*{Flows.} We now discuss the case of flows. Let $(X,d)$ be a metric space supporting a Borel probability measure $\mu$ and let $A := \{a_t \colon X \to X\}_{t \in \mathbb{R}}$ be a measure-preserving flow. Fix $m \in \mathbb{N}$. By an $m$-dimensional cocycle over $A$ we mean a map $C \colon X \times \mathbb{R} \to \mathrm{GL}(m,\mathbb{R})$ satisfying the cocycle identity 
\[
C(x,r+s) = C(a_r x,s) \cdot C(x,r) \quad \text{ for $\mu$-almost every } x\in X \text{ and every } r,s\in \mathbb{R}\,.
\]

Recall that $\langle \cdot, \cdot\rangle$ denotes the standard inner product on $\mathbb{R}^m$, that $\|\cdot\|$ denotes the corresponding Euclidean norm, and that $\nu$ denotes the induced probability measure on the projectivization $\mathbb{P}\mathbb{R}^m$. Recall that the projectivized bundle $X \times \mathbb{P}\mathbb{R}^m$ is endowed with the product measure $\mu \otimes \nu$.

Given $(x,v)$ in $ X \times (\mathbb{R}^m \setminus \{0\})$ or $X \times \mathbb{P}\mathbb{R}^m$ and $t \in \mathbb{R}$ consider the quantities
\begin{align*}
    \sigma(x,v,t) := \log \frac{\|C(x,t) v\|}{\|v\|} \in \mathbb{R},\\
    \sigma(x,t) := \log \sup_{v \in \mathbb{PR}^m} \frac{\|C(x,t)v\|}{\|v\|} \in \mathbb{R}.
\end{align*}
We say the cocycle $C$ is log-integrable if the following maps belong to $L^1(X,\mu)$,
\[
x \mapsto \max\{0,\sigma(x,1)\}, \quad x \mapsto \max\{0,\sigma(x,-1)\}.
\]

We say the cocycle $C$ satisfies a (spatial) distributional limit theorem (DLT) on $(X \times \mathbb{PR}^m, \mu \otimes \nu)$ if there exists a real function $\A:=(A_t)_{t \in \mathbb{R}}$, a real positive function $\V:=(V_t)_{t \in \mathbb{R}}$ with $V_t \to \infty$ as $t \to \infty$, and a random variable $S$, such that the random variables 
$$
S_t(x,v) = \frac{\sigma(x,v,t) -A_t}{ V_t } \ \ \text{ on } (X\times \R^m, \mu\otimes \nu)
$$
converge in distribution to $S$ as $t \to \infty$, i.e., for every interval $(a,b) \subseteq \mathbb{R}$ such that $\mathbb{P}(S \in \{a,b\}) = 0$, the following holds,
\[
\lim_{t \to \infty} (\mu \otimes \nu)(\{(x,v) \in X \in \mathbb{PR}^m \ | \ S_t(x,v) \in (a,b) \}) = \mathbb{P}(S \in (a,b)).
\]
We usually refer to $\A$ as the averaging sequence, $\V$ as the normalizing sequence, and $S$ as the limiting distribution.

We say the cocycle $C$ is $\V$-sufficiently-bounded if for $\mu$-almost-every $x \in X$, $\nu$-almost-every $v \in \mathbb{PR}^m$, and every $t \in \mathbb{R}$, the following holds,
\[
|\sigma(x,v,t) - \sigma(a_1x,C(x,1)v,t)| = o_{x,v}(V_t) .
\]

\begin{remark}
In applications, see \S\ref{S:KZ}, it is common for the cocycle $C$ to be bounded in the following stronger sense: for $\mu$-almost-every $x \in X$ and $\nu$-almost-every $v \in \mathbb{PR}^m$, 
\[
|\sigma(x,v,1)| = O(1).
\]
\end{remark}

We say the cocycle $C$ has $\V$-simple-dominated-splitting  if, for $\mu$-almost every $x \in X$ and $\nu$-almost every $v,w \in \mathbb{PR}^m$, the following holds
    \[
    |\sigma(x,v,t) - \sigma(x,w,t)| = o_{x,v,w}(V_t).
    \]

   Assume $C$ is log-integrable with simple top Lyapunov exponent $\lambda \in \mathbb{R}$. We say a pair $(x,v) \in X \times \mathbb{R}^m$ is future-Oseledets-generic if
\[
\lim_{t \to \infty} \frac{\sigma(x,v,t)}{t} = \lambda.
\]

In applications we use the following result; its proof is analogous to that of Lemma \ref{lemma:exp} and its details are left to the reader.

\begin{lemma}
    \label{lemma:exp_flow}
    Let $(X,\mathcal{B},\mu)$ be a probability space and $A := \{a_t \colon X \to X\}_{t \in \mathbb{R}}$ be a measure-preserving, ergodic flow.  Suppose $ C \colon X \times \mathbb{R} \to \mathrm{GL}(\mathbb{R},m)$ is a measurable cocycle over $T$. Assume $C$ is log-integrable with simple top Lyapunov exponent $\lambda \in \mathbb{R}$. Then, for every $x\in X$, every $v,w \in \mathbb{R}^m$ such that the pairs $(x,v),(x,w) \in X \times \mathbb{R}^m$ are future-Oseledets-generic, and every $t \in \mathbb{R}$, the following estimate holds,
    \[
    |\sigma(x,v,t) - \sigma(x,w,t)| = O_{x,v,w}(1)\,.
    \]
    In particular, $C$ has simple dominated splitting with respect to any diverging sequence $\V$.
    \end{lemma}


Now let $U \colon X \to X$ be a measure-preserving, ergodic transformation and let $D \colon X \times \mathbb{N} \to \mathrm{GL}(m,\mathbb{R})$ be a measurable cocycle over $U$. Recall that we say the pair $(A,U)$ is contracting if for $\mu$-almost-every $x \in X$,
\[
\lim_{t\to \infty} d(a_t U x, a_t x) = 0.
\]

We say the cocycles $(C,D)$ are $(A,U, \V)$-adapted if for $\mu$-almost-every $x \in X$, and $\nu$-almost-every $v \in \mathbb{PR}^m$, the following estimate holds,
\[
|\sigma(x,v,t) - \sigma(Ux,D(x,1)v,t)| = o_{x,v}(V_t).
\]

\begin{remark}
In applications, see \S \ref{S:KZ}, it is common to have the following stronger condition: for $\mu$-almost-every $x \in X$, $\nu$-almost every $v \in \mathbb{PR}^m$, 
\[
|\sigma(x,v,t) - \sigma(Ux,D(x,1)v,t)| = O_{x,v}(1).
\]
\end{remark}

The main result for flows is the following. Its proof is analogous to that of Theorem \ref{theo:cocycle_MCLT} and its details are left to the reader.

\begin{theorem}
\label{theo:cocycle_MCLT_flows}
Let $(X,d)$ be a metric space supporting a Borel probability measure $\mu$, let $A := \{a_t \colon X \to X\}_{t \in \mathbb{R}}$ a measure-preserving flow, and let $C \colon X \times \mathbb{R} \to \mathrm{GL}(m,\mathbb{R})$  be a measurable cocycle over $A$. Assume that $A$ satisfies a DLT on $(X \times \mathbb{PR}^m, \mu \otimes \nu)$ with averaging function $\A :=(A_t)_{t \in \mathbb{R}}$, normalizing function 
$\V=(V_t)_{t \in \mathbb{R}}$, and limiting distribution $S$ in the sense that the random variables
\[
S_t(x,v) := \frac{\sigma(x,v,t) - A_t}{V_t} 
\ \text{ on } (X\times \R^m, \mu\otimes \nu)
\]
converge in distribution to $S$ as $t \to \infty$, i.e., for every interval $(a,b) \subseteq \R$ such that $\mathbb{P}(S \in \{a,b\}) = 0$, the following holds,
\[
\lim_{t \to \infty}(\mu \otimes \nu)(\{(x,v) \in X \times \mathbb{PR}^m \ | \ S_t(x,v) \in (a,b)\}) = \mathbb{P}\left(S \in (a,b) \right).
\]
Assume that $C$ is $\V$-sufficiently-bounded and has a $\V$-simple-dominated-splitting. Assume in addition that there exist a measure-preserving, ergodic transformation $U \colon X \to X$ and a measurable cocycle  $D \colon X \times \mathbb{N} \to \mathrm{GL}(m,\mathbb{R})$ over $U$ such  that $(A,U)$ is contracting and such that $(C,D)$ is $(A,U, \V)$-adapted. Then, the above DLT  holds in the mixing sense, i.e. for every pair of Borel measurable subsets $B,E \subseteq X$ and every interval $(a,b) \subseteq \R$ such that $\mathbb{P}(S \in \{a,b\}) = 0$, the following holds,
\begin{gather*}
  \lim_{t \to \infty}(\mu \otimes \nu)(\{(x,v) \in X \times \mathbb{PR}^m \ | \ x \in B, \ S_t(x,v) \in (a,b), \ a_t x \in E\}) \\
  = \mu(B) \cdot \mathbb{P}\left(S \in (a,b) \right) \cdot \mu(E).
\end{gather*}
\end{theorem}


We also highlight the following result of independent interest;  its proof is analogous to that of Theorem \ref{theo:cocycle2_CCLT} and its details are left to the reader.

\begin{theorem}
\label{theo:cocycle2_CCLT_flows}
Let $(X,\mathcal{B},\mu)$ be a probability space, let $A := \{a_t \colon X \to X\}_{t \in \mathbb{R}}$ a measure-preserving, ergodic flow, and let $C \colon X \times \mathbb{R} \to \mathrm{GL}(m,\mathbb{R})$  be a measurable cocycle over $A$. Assume that $A$ satisfies a DLT on $(X \times \mathbb{PR}^m, \mu \otimes \nu)$ with averaging function $\A :=(A_t)_{t \in \mathbb{R}}$, normalizing function 
$\V=(V_t)_{t \in \mathbb{R}}$, and limiting distribution $S$ in the sense that the random variables
\[
S_t(x,v) := \frac{\sigma(x,v,t) - A_t}{V_t} 
\ \text{ on } (X\times \R^m, \mu\otimes \nu)
\]
converge in distribution to $S$ as $t \to \infty$, i.e., for every interval $(a,b) \subseteq \R$ such that $\mathbb{P}(S \in \{a,b\}) = 0$, the following holds,
\[
\lim_{t \to \infty}(\mu \otimes \nu)(\{(x,v) \in X \times \mathbb{PR}^m \ | \ S_t(x,v) \in (a,b)\}) = \mathbb{P}\left(S \in (a,b) \right).
\]
Assume in addition that $C$ is $\V$-sufficiently-bounded and has $\V$-simple-dominated-splitting. Then, the above DLT holds in the conditional sense, i.e. for every $B \in \mathcal{B}$ and every interval $(a,b) \subseteq \R$ such that $\mathbb{P}(S \in \{a,b\}) = 0$, the following holds
\begin{gather*}
\lim_{t \to \infty}(\mu \otimes \nu)(\{(x,v) \in X \times \mathbb{PR}^m \ | \ x \in B, \ S_t(x,v) \in (a,b)\}) = \mu(B) \cdot \mathbb{P}\left(S \in (a,b) \right).
\end{gather*}
\end{theorem}

Let $(X,\mathcal{B},\mu)$ be a probability space, let $A := \{a_t \colon X \to X\}_{t \in \mathbb{R}}$ be a measure-preserving, ergodic flow, and let $ C \colon X \times \mathbb{R} \to \mathrm{GL}(\mathbb{R},m)$ be a measurable cocycle over $A$. Consider a real positive function $\V := (V_t)_{t \in \mathbb{R}}$ with $V_t \to \infty$ as $t \to \infty$. We say the cocycle $C$ has $\V$-strong-simple-dominated-splitting  if for $\mu$-almost-every $x \in X$, $\nu$-almost-every $v \in \mathbb{PR}^m$, and every $t \in \mathbb{R}$, the following estimate holds: 
    \[
    |\sigma(x,v,t) - \sigma(x,t)| = o_{x,v}( V_t).
    \]

In applications we use the following result; its proof is analogous to that of Lemma \ref{lemma:exp2} and its details are left to the reader.

\begin{lemma}
    \label{lemma:exp2_flow}
    Let $(X,\mathcal{B},\mu)$ be a probability space and $A := \{a_t \colon X \to X\}_{t \in \mathbb{R}}$ be a measure-preserving, ergodic flow.  Suppose $ C \colon X \times \mathbb{R} \to \mathrm{GL}(\mathbb{R},m)$ is a measurable cocycle over $T$. Assume $C$ is log-integrable with simple top Lyapunov exponent $\lambda \in \mathbb{R}$. Then, for every $x \in X$, every $v \in \mathbb{PR}^m$ such that the pair $(x,v) \in X \times \mathbb{R}^m$ is future-Oseledets-generic, and every $t \in \mathbb{R}$, the following estimate holds,
    \[
    |\sigma(x,v,t) - \sigma(x,t)| = O_{x,v}(1).
    \]
    In particular, $C$ has a strong simple dominated splitting with respect to any sequence $\mathcal{V}$.
\end{lemma}
    

We also highlight the following results for the operator norm of cocycles over flows. Their proofs are analogous to those of Theorems \ref{theo:cocycle_CCLT_norm} and \ref{theo:cocycle_CCLT_norm2}; details are left to the reader.

\begin{theorem}
\label{theo:cocycle_CCLT_norm_flows}
Let $(X,\mathcal{B},\mu)$ be a probability space, let $A := \{a_t \colon X \to X\}_{t \in \mathbb{R}}$ be a measure-preserving flow, and let $C \colon X \times \mathbb{R} \to \mathrm{GL}(m,\mathbb{R})$  be a measurable cocycle over $A$. Assume that $A$ satisfies a DLT on $(X \times \mathbb{PR}^m, \mu \otimes \nu)$ with averaging function $\A :=(A_t)_{t \in \mathbb{R}}$, normalizing function 
$\V=(V_t)_{t \in \mathbb{R}}$, and limiting distribution $S$ in the sense that the random variables
\[
S_t(x,v) := \frac{\sigma(x,v,t) - A_t}{V_t} 
\ \text{ on } (X\times \R^m, \mu\otimes \nu)
\]
converge in distribution to $S$ as $t \to \infty$, i.e., for every interval $(a,b) \subseteq \R$ such that $\mathbb{P}(S \in \{a,b\}) = 0$, the following holds,
\[
\lim_{t \to \infty}(\mu \otimes \nu)(\{(x,v) \in X \times \mathbb{PR}^m \ | \ S_t(x,v) \in (a,b)\}) = \mathbb{P}\left(S \in (a,b) \right).
\]
Assume in addition that $C$ has $\V$-strong-simple-dominated-splitting  Then, $C$ satisfies the above DLT with respect to the operator norm in the sense that the random variables
\[
S_t(x) := \frac{\sigma(x,t)-A_t}{V_t} 
\ \text{ on } (X, \mu)
\]
converge in distribution to $S$ as $t \to \infty$, i.e., for every interval $(a,b) \subseteq \R$ such that $\mathbb{P}(S \in \{a,b\}) = 0$, the following holds,
\begin{gather*}
  \lim_{t \to \infty} \mu(S_t(x) \in (a,b))
  = \mathbb{P}\left(S \in (a,b) \right).
\end{gather*}
\end{theorem}

\begin{theorem}
\label{theo:cocycle_CCLT_norm2_flows}
Let $(X,\mathcal{B},\mu)$ be a probability space, let $A := \{a_t \colon X \to X\}_{t \in \mathbb{R}}$ be a measure-preserving flow, and let $C \colon X \times \mathbb{R} \to \mathrm{GL}(m,\mathbb{R})$  be a measurable cocycle over $A$. Assume that $A$ satisfies a mixing DLT on $(X \times \mathbb{PR}^m, \mu \otimes \nu)$ with averaging function $\A :=(A_t)_{t \in \mathbb{R}}$, normalizing function 
$\V=(V_t)_{t \in \mathbb{R}}$, and limiting distribution $S$ in the sense that for the random variables
\[
S_t(x,v) := \frac{\sigma(x,v,t) - A_t}{V_t} 
\ \text{ on } (X\times \R^m, \mu\otimes \nu),
\]
for every pair of measurable subsets $B,D \in \mathcal{B}$, and for every interval $(a,b) \subseteq \R$ such that $\mathbb{P}(S \in \{a,b\}) = 0$, the following holds,
\begin{gather*}
\lim_{t \to \infty}(\mu \otimes \nu)(\{(x,v) \in X \times \mathbb{PR}^m \ | \ x \in B, \ S_t(x,v) \in (a,b), a_t x \in D\}) \\
= \mu(B) \cdot \mathbb{P}\left(S \in (a,b) \right) \cdot \mu(D).
\end{gather*}
Assume in addition that the cocycle $C$ has $\V$-strong-simple-dominated-splitting.  Then, $C$ satisfies the above mixing DLT with respect to the operator norm in the sense that for the random variables
\[
S_t(x) := \frac{\sigma(x,t)-A_t}{V_t} 
\ \text{ on } (X, \mu),
\]
for every pair of measurable subsets $B,D \in\mathcal{B}$, and for every interval $(a,b) \subseteq \R$ such that $\mathbb{P}(S \in \{a,b\}) = 0$, the following holds,
\begin{gather*}
\lim_{t \to \infty} \mu(\{ x \in X  \ | \ x \in B, \ S_t(x) \in (a,b), a_t x \in D\}) \\
= \mu(B) \cdot \mathbb{P}\left(S \in (a,b) \right) \cdot \mu(D).
\end{gather*}
\end{theorem}

Let $(X,\mathcal{B},\mu)$ be a probability space, let $A := \{a_t \colon X \to X\}_{t \in \mathbb{R}}$ be a measure-preserving, ergodic flow, and let $ C \colon X \times \mathbb{R} \to \mathrm{GL}(\mathbb{R},m)$ be a measurable cocycle over $A$. Consider a  positive function $\V := (V_t)_{t \in \mathbb{R}}$ with $V_t \to \infty$ as $t \to \infty$. We say a measurable section $s \colon X \to \mathbb{R}^m$ is $(C,\mathcal{V})$-generic if for $\mu$-almost-every $x \in X$, $\nu$-almost every $w \in \mathbb{PR}^m$, and every $t \in \mathbb{R}$, the following estimate holds,
    \[
    |\sigma(x,v(x),t) - \sigma(x,w,t)| = o_{x,w}( V_t).
    \]

    In applications we use the following result; its proof is analogous to that of Lemma \ref{lemma:exp3} and its details are left to the reader.

\begin{lemma}
    \label{lemma:exp3_flow}
   Let $(X,\mathcal{B},\mu)$ be a probability space and $A := \{a_t \colon X \to X\}_{t \in \mathbb{R}}$ be a measure-preserving, ergodic flow.  Suppose $ C \colon X \times \mathbb{R} \to \mathrm{GL}(\mathbb{R},m)$ is a measurable cocycle over $T$. Assume $C$ is log-integrable with simple top Lyapunov exponent $\lambda \in \mathbb{R}$. Let $s \colon X \to \mathbb{R}^m$ be a measurable section such that $(x,v(x)) \in X \times \mathbb{R}^m$ is future-Oseledets-generic for $\mu$-almost every $x \in X$. Then, for every $x \in X$, every $w \in \mathbb{PR}^m$ such that $(x,w) \in X \times \mathbb{R}^m$ is future-Oseledets-generic, and every $N \in \mathbb{N}$,
    \[
    |\sigma(x,v(x),N) - \sigma(x,w,N)| = O_{x}(1).
    \]
    In particular, the section $s \colon X \to \mathbb{R}^m$ is $(C,\mathcal{V})$-generic for any diverging sequence $\mathcal{V}$.
\end{lemma}
    

We also highlight the following results for generic sections of cocycles. Their proofs are analogous to those of Theorems \ref{theo:cocycle_CCLT_sec} and \ref{theo:cocycle_CCLT_sec2}; details are left to the reader.

\begin{theorem}
\label{theo:cocycle_CCLT_sec_flows}
Let $(X,\mathcal{B},\mu)$ be a probability space, let $A := \{a_t \colon X \to X\}_{t \in \mathbb{R}}$ be a measure-preserving flow, and let $C \colon X \times \mathbb{R} \to \mathrm{GL}(m,\mathbb{R})$  be a measurable cocycle over $A$. Assume that $A$ satisfies a DLT on $(X \times \mathbb{PR}^m, \mu \otimes \nu)$ with averaging function $\A :=(A_t)_{t \in \mathbb{R}}$, normalizing function 
$\V=(V_t)_{t \in \mathbb{R}}$, and limiting distribution $S$ in the sense that the random variables
\[
S_t(x,v) := \frac{\sigma(x,v,t) - A_t}{V_t} 
\ \text{ on } (X\times \R^m, \mu\otimes \nu)
\]
converge in distribution to $S$ as $t \to \infty$, i.e., for every interval $(a,b) \subseteq \R$ such that $\mathbb{P}(S \in \{a,b\}) = 0$, the following holds,
\[
\lim_{t \to \infty}(\mu \otimes \nu)(\{(x,v) \in X \times \mathbb{PR}^m \ | \ S_t(x,v) \in (a,b)\}) = \mathbb{P}\left(S \in (a,b) \right).
\]
Let $s \colon X \to \mathbb{R}$ be a measurable, $(C,\mathcal{V})$-generic section. Then, $s$ satisfies the above DLT in the sense that the random variables
\[
S_t(x) := \frac{\sigma(x,v(x),t)-A_t}{V_t} 
\ \text{ on } (X, \mu)
\]
converge in distribution to $S$ as $t \to \infty$, i.e., for every interval $(a,b) \subseteq \R$ such that $\mathbb{P}(S \in \{a,b\}) = 0$, the following holds,
\begin{gather*}
  \lim_{t \to \infty} \mu(S_t(x) \in (a,b))
  = \mathbb{P}\left(S \in (a,b) \right).
\end{gather*}
\end{theorem}

\begin{theorem}
\label{theo:cocycle_CCLT_sec2_flows}
Let $(X,\mathcal{B},\mu)$ be a probability space, let $A := \{a_t \colon X \to X\}_{t \in \mathbb{R}}$ be a measure-preserving flow, and let $C \colon X \times \mathbb{R} \to \mathrm{GL}(m,\mathbb{R})$  be a measurable cocycle over $A$. Assume that $A$ satisfies a mixing DLT on $(X \times \mathbb{PR}^m, \mu \otimes \nu)$ with averaging function $\A :=(A_t)_{t \in \mathbb{R}}$, normalizing function 
$\V=(V_t)_{t \in \mathbb{R}}$, and limiting distribution $S$ in the sense that for the random variables
\[
S_t(x,v) := \frac{\sigma(x,v,t) - A_t}{V_t} 
\ \text{ on } (X\times \R^m, \mu\otimes \nu),
\]
for every pair of measurable subsets $B,D \in \mathcal{B}$, and for every interval $(a,b) \subseteq \R$ such that $\mathbb{P}(S \in \{a,b\}) = 0$, the following holds,
\begin{gather*}
\lim_{t \to \infty}(\mu \otimes \nu)(\{(x,v) \in X \times \mathbb{PR}^m \ | \ x \in B, \ S_t(x,v) \in (a,b), a_t x \in D\}) \\
= \mu(B) \cdot \mathbb{P}\left(S \in (a,b) \right) \cdot \mu(D).
\end{gather*}
Let $s \colon X \to \mathbb{R}$ be a measurable, $(C,\mathcal{V})$-generic section.  Then, $s$ satisfies the above mixing DLT in the sense that for the random variables
\[
S_t(x) := \frac{\sigma(x, v(x),t)-A_t}{V_t} 
\ \text{ on } (X, \mu),
\]
for every pair of measurable subsets $B,D \in\mathcal{B}$, and for every interval $(a,b) \subseteq \R$ such that $\mathbb{P}(S \in \{a,b\}) = 0$, the following holds,
\begin{gather*}
\lim_{t \to \infty} \mu(\{ x \in X  \ | \ x \in B, \ S_t(x) \in (a,b), a_t x \in D\}) \\
= \mu(B) \cdot \mathbb{P}\left(S \in (a,b) \right) \cdot \mu(D).
\end{gather*}
\end{theorem}

\section{The Kontsevich-Zorich cocycle} \label{S:KZ}

\subsection*{Outline of this section.} In this section we apply the main theorems of the previous section to the Kontsevich--Zorich cocycle. More precisely, we derive distributional mixing laws of large numbers and distributional mixing central limit theorems for exterior powers of $\mathrm{SL}(2,\mathbb{R})$-invariant subbundles of the Kontsevich--Zorich cocycle under natural, well studied conditions; see Theorems \ref{theo:LLN1}, \ref{theo:LLN2}, \ref{theo:LLN3}, \ref{theo:CLT1}, \ref{theo:CLT2}, and \ref{theo:CLT3} for precise statements. We begin with a brief overview of the rich theory of the Kontsevich--Zorich cocycle, including many important recent developments. We then proceed to prove the desired results using these developments and the technology introduced in the previous section. In terms of examples, an important emphasis is placed on the study of the invariant part of the Kontsevich--Zorich cocycle over loci of orientation double covers of quadratic differentials on principal strata. These examples are of crucial importance for the study of the statistics of the action in (co)homology of mapping class groups carried out in \cite{statistics}.

\subsection*{The Kontsevich--Zorich cocycle.} The Kontsevich--Zorich (KZ) cocycle, introduced in 
\cite{Kon97,KZ97} is arguably the central
object of study in Teichm\"uller dynamics. 
By definition, it keeps track of the homology of
trajectories of translation flows under the renormalization dynamics given by the Teichm\"uller geodesic flow, and, at the same time, it contains all
the essential dynamical information of the tangent cocycle of the latter. The simplicity of the top exponent of the KZ cocycle implies, on one hand, the unique ergodicity of almost all translation flows, and, on the other hand, the non-uniform hyperbolicity of the Teichm\"uller geodesic flow (with respect to any ergodic invariant measure).

The KZ cocycle is a cocycle over the Teichm\"uller geodesic flow on the moduli space of Abelian (or quadratic) holomorphic differentials. We recall that  there exists a natural action of the group $\text{\rm SL}(2,\mathbb{R})$ on the Teichm\"uller space of Abelian differentials of unit total area on Riemann surfaces of genus $g \geq 1$ which descends to an action on their moduli space $\smash{\mathcal H_g}$. The $\text{\rm SL}(2,\mathbb{R})$-action is defined as follows. Let $A \in \text{\rm SL}(2,\mathbb{R})$ and $\omega$ be a holomorphic  Abelian differential, i.e., a holomorphic $1$-form, on a Riemann surface $M$. Then, the differential $A \cdot \omega$
is defined by the condition that
\begin{equation}
\label{eq:sl2r}
\left(\begin{array}{c} \mathrm{Re} (A \cdot \omega) \\ \mathrm{Im} (A \cdot \omega) \end{array}\right) := 
A \cdot \left(\begin{array}{c} \mathrm{Re} (\omega) \\ \mathrm{Im} (\omega) \end{array} \right)\,.
\end{equation}
Since the action on the right hand side of \eqref{eq:sl2r} is linear, it is clear that the above formula defines a closed complex valued $1$-form. It can then be proved that 
$A\cdot \omega$ is holomorphic with respect to a unique complex structure on the underlying surface of genus $g \geq 1$.

It is well-known that the above action respects the strata $\mathcal H(\kappa) \subseteq \mathcal{H}_g$, where $\kappa:=\kappa=(k_1, \dots, k_\sigma)$ is an integral partition of $2g-2$ describing the multiplicities of the zeroes of the differentials considered, and their connected components. Furthemore, the above action respects the natural affine structure on strata
induced by period coordinates and preserves the canonical 
Masur--Veech probability measures \cite{Mas82,Vee82}. 

The one-parameter subgroups $A := \{a_t\}_{t \in \mathbb{R}}$ and $H^\pm := \{h_t^\pm\}_{t \in \mathbb{R}}$
given for $t \in \mathbb{R}$ by
$$
g_t := \left(\begin{array}{c c} e^t & 0 \\ 0 & e^{-t}\end{array} \right) \,, \quad h^+_t := \left(\begin{array}{c c} 1 & t \\ 0 & 1\end{array}\right)\,, \quad h^-_t := \left( \begin{array}{c c} 1 & 0 \\ t & 1\end{array} \right)
$$
induce, via the $\mathrm{SL}(2,\mathbb{R})$-action described above, the so called Teichm\"uller geodesic flow and
horocycle flows on connected components of strata. These flows were originally introduced by Masur in \cite{Mas82,Mas85}. 

The relative Kontsevich--Zorich cocycle $\mathrm{KZ}$ can be defined as the flow on the relative cohomology bundle $H^1_\kappa(M, \Sigma, \mathbb{R})$  over
a stratum $\mathcal H(\kappa)$ of Abelian diffentials with zeros on a finite set $\Sigma \subseteq M$
given by parallel transport of cohomology classes
with respect to the Gauss--Manin connection along orbits of the Teichm\"uller geodesic flow \cite{Kon97, KZ97}. More concretely, it can be defined as follows. Let $\smash{\widehat {\mathcal H}(\kappa)}$ denote the stratum of the Teichm\"uller space above the stratum $\mathcal H(\kappa)$ of the correponding moduli space. Since $\smash{\widehat {\mathcal H}(\kappa)}$ is simply connected, the cohomology bundle can be
trivialized. Hence, $\mathrm{KZ}$ can be defined as
the projection of the trivial cocycle
$$
A \times \text{id} : \widehat {\mathcal H}(\kappa) \times H^1(M, \Sigma, \R) \to \widehat{\mathcal H}(\kappa) \times H^1(M, \Sigma, \R)
$$
under the natural action of the mapping class group $\Gamma_g$ on $\smash{\widehat {\mathcal H}(\kappa)}$.
It is therefore a cocycle on the (orbilfold)
vector bundle $\smash{\widehat{\mathcal H}(\kappa)} \times H^1(M, \Sigma, \R)/\Gamma_g$ over the stratum of the moduli space $\mathcal H(\kappa) = \smash{\widehat{\mathcal H}(\kappa)} /\Gamma_g$. The relative
KZ cocycle projects to a well-defined symplectic cocycle, known as the absolute Kontsevich--Zorich (KZ) cocycle, on the absolute cohomology bundle with fibers $H^1(M,\R)$ through the natural forgetful projection 
\[
H^1(M, \Sigma,\R) \to H^1(M, \R).
\]

\subsection*{The Lyapunov spectrum.} The crucial properties of the relative and absolute KZ cocycles are related to their Lyapunov spectra. Since the moduli space of Abelian differentials is not compact, the choice of the norm on the cohomology bundle matters.  Following the insight of \cite{Kon97, KZ97} the Hodge norm on the absolute cohomology $H^1(M, \R)$ (which can be extended to relative cohomology) makes the analysis of the
Lyapunov structure amenable to methods in complex analysis. 

\begin{theorem} 
\label{thm:spectral_gap} \cite{Kon97, KZ97, For02} The Lyapunov spectrum of the relative Kontsevich--Zorich cocycle with respect to any $A$-invariant ergodic probability measure $\mu$ on a stratum $\mathcal{H}(\kappa) \subseteq \mathcal{H}_g$ with $\kappa := (\kappa_1,\dots,\kappa_\sigma)$ is well-defined and has the form 
$$
\lambda_1^\mu =1 > \lambda_2^\mu \geq \dots\geq \lambda_g^\mu \geq 0= \dots =0 \geq -\lambda_g^\mu \geq \dots \geq -\lambda_2^\mu > - \lambda_1^\mu =-1\,.
$$
The exponents $\lambda_1^\mu, \dots, \lambda^\mu_g, -\lambda_g^\mu, \dots, -\lambda_1^\mu$ are the Lyapunov exponents of the absolute Kontsevich--Zorich cocycle.
The exponents $\pm \lambda_1^\mu= \pm 1$
are carried by the so called tautological subbundle, i.e. the real $2$-dimensional span of the real and imaginary parts of the underlying Abelian differential.
The exponent $0$ appears with minimum multiplicity
equal to the $\sigma -1$ on the kernel of the forgetful map $H^1(M, \Sigma, \R) \to H^1(M, \Sigma)$. 
\end{theorem}

\subsection*{The simplicity conjectures.}  It was conjectured by Zorich~\cite{Zor94, Zor99}, and by Kontsevich and Zorich \cite{Kon97,KZ97}, that the KZ spectrum is
simple for Masur--Veech measures on strata of Abelian and quadratic differentials. For Abelian differentials this conjecture was proved by A. Avila and M. Viana \cite{AV07}. Indeed, we have the following.

\begin{theorem} \cite{AV07}  \label{theo:sim}
The Kontsevich-Zorich
spectrum for all Masur--Veech measures $\mu$ on (connected components of) strata $\mathcal M$ of Abelian differentials on surfaces of genus $g \geq 1$ is simple, i.e., the following strict inequalities hold,
$$
\lambda_1^{\mathcal M} > \lambda_2^{\mathcal M} > \dots  > \lambda_g^{\mathcal M} >0\,.
$$
\end{theorem}

\begin{remark}
The strict inequality $\lambda^{\mathcal M}_g >0$ (non-uniform hyperbolicity) as well as the case $g=2$ of Theorem \ref{theo:sim} were proved earlier in \cite{For02}.
\end{remark}

The analogous theorem for strata of quadratic differentials has proved harder to attain
and the proof of the Kontsevich--Zorich conjecture has actually been completed only recently. We recall that, by an  oriented double covering construction, which we summarize below, it is possible to view strata of quadratic differentials with at most simple poles as suborbifolds of strata of moduli spaces of Abelian differentials on surfaces of higher genus, where the genus depends on the multiplicities of the zeros and
the number of poles. 

\noindent The theory of the KZ spectrum for quadratic differentials is therefore a special
important case of the theory of the KZ spectrum for general $\text{\rm SL}(2, \R)$-invariant measures on moduli spaces of Abelian differentials.   

Given a (holomorphic) quadratic differential $q$ on a Riemann surface $M$, one can construct a canonical (orienting) double cover $\pi: \smash{\widehat M}  \to M$, branched only at the zeroes of $q$ of odd multiplicity, with $\smash{\widehat M}$ connected if and only if $q$ is not the square of an Abelian differential. Moreover, $\pi^\ast(q) = \omega^2$,
where $\omega$ is an Abelian differential on $\smash{\widehat M}$. This well-known  construction can be described as follows. Let $\{(\phi_\alpha,U_\alpha)\}_{\alpha\in A}$
be an atlas of canonical coordinates for $q$ on $M\setminus \{q=0\} $. For any $\alpha\in A$,  let 
$\pi^{-1}(U_\alpha) = \smash{\widehat U^+_\alpha} \cup \smash{\widehat U^-_\alpha}$ and let 
$$
\widehat{\phi}^\pm_\alpha (z) = \pm \sqrt{\phi_\alpha(z)}\,, \quad \text{ for  }  z\in \widehat U^\pm_\alpha\,.
$$
There exists a unique translation surface $(\smash{\widehat M}, \omega)$ such that the charts $(\smash{\widehat{\phi}^\pm_\alpha}, \smash{\widehat U^\pm_\alpha})$ give a translation atlas on $\smash{\widehat M} \setminus \{\omega=0\}$.  The surface $\smash{\widehat M}$ can be defined as the completion of
the Riemann surface quotient
$
\sqcup_{\alpha\in A}  \smash{\widehat U^\pm_\alpha} / \sim  \,
$
 with respect to the following equivalence relation:  $z \sim z'$
if and only if there exist $\alpha, \beta$
such that $U_\alpha \cap U_\beta \not =\emptyset$ with $z, z' \in \smash{\widehat U^+_\alpha}$ and $\smash{\widehat \phi^+_\alpha}(z) = \smash{\widehat \phi^+_\beta}(z) $ or $z, z' \in \smash{\widehat U^-_\alpha}$ and $\smash{\widehat \phi^-_\alpha}(z) = \smash{\widehat \phi^-_\beta}(z) $. 

\smallskip
By the construction we have that
$\pi^\ast (q) = \omega^2$.
It can be verified that the orienting double cover construction gives an embedding 
$$
\mathcal Q (n_1, \dots, n_\nu, n_{\nu+1}, \dots, n_\tau) \to \mathcal H \left(n_1+1, \dots, n_\nu+1, \frac{n_{\nu+1}}{2}, \frac{n_{\nu+1}}{2}, \dots, \frac{n_{\tau}}{2}, \frac{n_{\tau}}{2}\right)
$$
of the stratum $\mathcal Q (n_1, \dots, n_\nu, n_{\nu+1}, \dots, n_\tau)$ of quadratic differentials with zeros/poles of odd multiplicities $(n_1, \dots, n_\nu) \in \{ 2k-1\ \vert \ k\in \N\}^\nu$
and zeros of even multiplicities 
$(n_{\nu+1}, \dots, n_{\tau}) \in \{ 2k \ \vert \ k\in \N \setminus \{0\}\}^{\tau-\nu}$ into the stratum
of Abelian differentials with zeroes of multiplicities $n_i+1$
for $i\in \{1, \dots, \nu\}$ and $n_{j}/2, n_{j}/2$ for $j\in \{\nu+1, \dots, \tau\}$; a zero of multiplicity $0$ corresponds by definition to a marked point.

There exists an involution $I:\smash{\widehat M} \to \smash{\widehat M}$ such that $I^\ast \omega =-\omega$.  The map $\sigma$
exchanges points in each regular fiber of $\pi : \smash{\widehat M} \to M$, hence $I: \smash{\widehat U^\pm_\alpha} \to \smash{\widehat U^\mp_\alpha}$ for all $\alpha\in A$, 
and equals the identity map on the zeroes of $\omega$ that cover odd order zeroes of $q$. Such an involution
induces a splitting on the relative cohomology  of $\smash{\widehat M}$ into invariant/even and anti-invariant/odd subspaces,
\begin{equation}
\label{eq:even/odd}
H^1(\smash{\widehat M}, \{\omega=0\}; \R) = 
H^{1}_+(\smash{\widehat M}, \{\omega=0\}; \R) \oplus H^{1}_-(\smash{\widehat M}, \{\omega=0\}; \R) \,,
\end{equation}
which is just the splitting into eigenspaces of eigenvalues $\pm 1$
of the map 
$$I^* : H^1(\smash{\widehat M}, \{\omega=0\}; \R) \to H^1(\smash{\widehat M}, \{\omega=0\}; \R)\,.
$$
There is a similar (dual) splitting in homology,
$$
H_1(\smash{\widehat M}, \{\omega=0\}; \R) = 
H_{1}^+(\smash{\widehat M}, \{\omega=0\}; \R) \oplus H_{1}^-(\smash{\widehat M}, \{\omega=0\}; \R) \,.
$$
It follows by construction that the maps
$$
\begin{aligned}
\pi^\ast: H^1(M, \{q=0,\infty\}, \R) \to H^1_+ (\smash{\widehat M}, \{\omega=0,m\}; \R)\,, \\
\pi_\ast: H_1^+(\smash{\widehat M}, \{\omega=0,m\}; \R) \to H_1 (M, 
\{q=0,\infty\}; \R)\,,
\end{aligned}
$$
are isomorphisms of vector spaces; here $\omega = m$ represents marked points and $q = \infty$ represents poles. We note that the Abelian differential $\omega$ is anti-invariant, i.e.,
$$
I^\ast(\omega) =-\omega\,.
$$
In particular, in terms of cohomology,
$$
[\omega] \in H^1_-(\smash{\widehat M}, \{\omega=0\}; \C)\,.
$$

Let $g$ denote the genus of $M$
and $\smash{\widehat g}$ the genus of the orienting double cover $\smash{\widehat M}$.
By the construction of the orienting double cover, every mapping class (homeomorphism, diffeomorphism)  $f$ in the mapping class group $\Gamma_g$ of $M$ lifts to a unique mapping class $\smash{\widehat f}$ (homeomorphism, diffeomorphism) 
in the mapping class group $\Gamma_{\smash{\widehat g}}$ of $\smash{\widehat M}$
such that
$$
f \circ \pi = \pi \circ \smash{\widehat f} \quad \text{ on }  \smash{\widehat M}\,.
$$
The mapping class group $\Gamma_g$ can  be identified with the subgroup of the mapping class group  $\Gamma_{\smash{\widehat g}}$ of mapping classes that commute with the projection of the orienting double cover. It follows that, on every image suborbifold  $\smash{\widehat {\mathcal Q}} ({\bf n})$  of any stratum $\mathcal Q ({\bf n})$ of (non-orientable) quadratic differentials into $\mathcal H_{\smash{\widehat g}}$, the Kontsevich-Zorich cocycle preserves the splitting 
\eqref{eq:even/odd} of the cohomology bundle into even and odd components.

Thus, for strata of quadratic differentials with simple poles, the analysis of the KZ spectrum splits into the separate problems of understanding the even and the odd parts of the spectrum  with respect to the canonical Masur--Veech measure coming from the stratum of quadratic differentials. We note that both the even and the odd spectrum are relevant in applications. In particular, the even spectrum is related to the tangent dynamics of the Teichm\"uller geodesic flow on the correspoding stratum of quadratic differentials, while the odd spectrum is related to the equidistribution of the leaves of the (non-orientable) foliations of quadratic differentials in the stratum. 

Recently it has been proved
that the monodromies of the even and odd Kontsevich-Zorich cocycles are Zariski dense in the symplectic groups $\operatorname{Sp}(2g)$ and $\operatorname{Sp}(2(\smash{\widehat g}-g))$, hence both the 
even and the odd Kontsevich-Zorich Lyapunov spectra are simple.

\begin{theorem} \cite{Tre13, GR, flow_group}
\label{thm:BDGGRS} 
Both the even and odd Kontsevich--Zorich Lyapunov
spectra for the Masur--Veech measures on strata of quadratic differentials with at most simple poles are simple. More precisely, let $\mathcal Q$ be a stratum of
quadratic differentials on surfaces of genus $g_{\mathcal Q}$ with oriented double covers in
a suborbilfold $\mathcal M_{\mathcal Q}$ of the moduli space of Abelian differentials on surfaces of genus $g_{\mathcal M_{\mathcal Q}}$. Then,
$$
1> \lambda^{\mathcal M_{\mathcal Q},+}_1 > \dots > \lambda^{\mathcal M_{\mathcal Q},+}_{g_Q}>0
\quad \text{ and } \quad 1= \lambda^{\mathcal M_{\mathcal Q},-}_1 > \dots > \lambda^{\mathcal M_{\mathcal Q},-}_{g_{\mathcal M_{\mathcal Q}} - g_{\mathcal Q}}>0,
$$
where the $+$ and $-$ signs denote whether the exponents correspond to the even or odd parts of the cocycle. Furthermore, the monodromy groups of  the even and odd components of the KZ cocycles are Zariski-dense subgroups of their ambient symplectic groups.
\end{theorem} 

\begin{remark}
    The spectral gap part of the above theorem
    ($1> \lambda^{\mathcal M_{\mathcal Q},+}_1$ and
    $1> \lambda^{\mathcal M_{\mathcal Q},-}_2$) follows
    from the general spectral gap result in Theorem~\ref{thm:spectral_gap}. R.~Trevi\~no \cite{Tre13} derived that all of the above Lyapunov exponents are strictly positive using a general criterion in \cite{For11}. Guti\'errez--Romo proved Zariski density and 
    simplicity for several strata in his thesis \cite{GR}. The complete result was finally proved in \cite{flow_group}.
\end{remark}

\subsection*{The KZ spectrum for general suborbifolds}

\noindent We now turn to discuss the Lyapunov spectrum for $\text{\rm SL}(2, \R)$-invariant measures other than Masur--Veech measures. The work of Eskin--Mirzkhani \cite{EM18} and Eskin, Mirzakhani and Mohammadi \cite{EMM15} proved that all $\text{\rm SL}
(2, \R)$-orbit-closures are affine suborbifolds and all $\text{\rm SL}(2, \R)$-invariant probability ergodic measures are absolutely continuous measures supported on  affine suborbifolds. 

Filip \cite{Fil16} proved a Deligne semi-simplicity theorem for tensor powers of the  cohomology (Hodge) bundle with respect to the $\text{\rm SL}(2,\R)$-action (as well as a semi-simplicitity result for flat (locally constant) subbundles). We record below a statement which follows from Filip's work and is sufficient for our purposes.

\begin{theorem} \label{theo:filip} \cite[Theorem A.6]{EM18} \cite[Theorem 1.2]{Fil16} Let $E$ denote any exterior power of the real cohomology bundle $H^1_{\mathcal M}(M, \R)$ over an $\text{\rm SL}(2,\R)$-invariant orbifold $\mathcal M$. Then $E$ splits as a direct, Hodge orthogonal sum of irreducible $\text{\rm SL}(2,\R)$ subbundles, in the sense that there exist $\text{\rm SL}(2,\R)$-invariant subbundles $V_i$, vector spaces $W_i$ and
division algebras $A_i$ over the real numbers $\R$  such that 
$$
E \simeq \bigoplus_{i} V_i \otimes_{A_i} W_i\,.
$$
A similar result holds for the complexified Kontsevich--Zorich cocyle (on exterior powers of the complex cohomology bundle). 
\end{theorem}

By Theorem \ref{theo:filip}, it is enough to analyze the Lyapunov spectrum on irreducible subbundles. 
In fact, Filip also proved that measurable 
$\text{\rm SL}(2,\R)$-subbundles on affine $\text{\rm SL}(2,\R)$-invariant suborbifolds are
always continuous, moreover, polynomial. 

\begin{theorem} \cite[Theorem 1.4]{Fil16} Let 
$E$ denote any exterior power of the real cohomology bundle $H^1_{\mathcal M}(M, \R)$ over an $\text{\rm SL}(2,\R)$-invariant orbifold $\mathcal M$. Let $V \subset E$ be  a measurable  $\text{\rm SL}(2,\R)$-invariant subbundle on $\mathcal M$. Then $V$
has a complement $V^\perp$ such that in local period coordinates the operator of projection to $V^\perp$ is polynomial.
\end{theorem} 

\noindent This result implies that, up to passing to a finite cover, it is enough to consider  the KZ cocycle on  strongly irreducible subbundles in the following sense. 

\begin{definition} \label{def:strongly_irr}  \cite[Definition 1.4]{CE15} 
A cocycle $A:G\times V \to V$ on a vector bundle $V$ over the action of a group $G$ on an ergodic  measure space $(X, \mu)$ is strongly irreducible  if it does not admit a $\mu$-measurable almost invariant splitting, that is, a measurable finite collection of subbundles $W_1, \dots, W_n \subseteq V$ such that $W_i \cap W_j=\{0\}$ for all $i\not=j \in \{1, \dots, n\}$ and such that,  for all $i\in \{1, \dots, n\}$  there exists $j\in \{1, \dots, n\}$ with 
$$
A(g,x) W_i(x) = W_j (gx) \,, \quad \text{ for all } (g,x) \in G\times X\,.
$$
\end{definition}

\noindent In fact, by \cite[Proposition B.4]{EFW18}, derived from \cite[Theorem 7.7]{Fil16}, there exists a (local) labeling of the measurable subbundles in Definition~\ref{def:strongly_irr} such that each subbundle is  continuous, in fact polynomial. It is therefore possible to define a finite cover of the orbifold on which the subundles are invariant, and therefore by a finite iteration of this procedure, to derive that  any $\mathrm{SL}(2, \R)$-invariant subbundle has an
equivariant splitting into strongly irreducible
components.

For a general $\text{\rm SL}(2,\R)$-invariant measure the Lyapunov spectrum can have multiplicities and zero exponents.  For instance the KZ spectrum is well-understood in the class of square tiled cyclic covers \cite{FMZ11}, which includes the two examples of maximally degenerate spectrum (the Teichm\"uller curves of the so called Eierlegende Wollmilchsau surface in genus $3$ and of the so called Platypus surface in genus $4$). A general result  of Filip \cite[Theorem 1.2]{Fil17} classifies the possible monodromies (hence the patterns of zero exponents which can arise) and gives a lower 
bound on the number of strictly positive (or 
non-zero) exponents for an arbitrary 
$\text{\rm SL}(2,\R)$-invariant suborbifold. 

\begin{theorem} \label{thm:Filip} \cite[Corollary 1.3]{Fil17} 
Let $\mathcal M$ be an affine invariant suborbifold of a stratum of Abelian differentials. Let $p(T\mathcal M)$ be the subbundle of the cohomology bundle given by the projection into absolute homology of the tangent subbundle $T\mathcal M$ of $\mathcal M$ (seen as subbundle of the relative cohomology bundle via period coordinates). Then:
\begin{itemize} 
\item   
no zero exponents can occur in $p(T\mathcal M)$, or any of its Galois conjugates;
\item  the Zariski closure of the monodromy in $p(T \mathcal M)$ or any of its Galois conjugates is the corresponding full symplectic group.
\end{itemize} 
\end{theorem}

The above theorem subsumes several earlier results on the KZ spectrum \cite{BM10, For11, FMZ11, FMZ14}. In the case of Teichm\"uller curves of square-tiled surfaces 
a criterion for the simplicity of the KZ spectrum
was given in \cite{MMY15}. 

\begin{corollary} Let $\mathcal Q$ be any stratum of
quadratic differentials and let  $\mathcal M_{\mathcal Q}$ be the corresponding suborbifold of orienting double covers in the corresponding moduli space of Abelian differentials.
Then, the even and odd subbundles $H^1_+(M, \R)$ and $H^1_-(M, \R)$ of the real Hodge bundle $H^1_{\mathcal M_{\mathcal Q}}(M, \R)$ over $\mathcal M_{\mathcal Q}$ are strongly irreducible invariant subbubdles for the Kontsevich--Zorich cocycle
over the Teichm\"uller flow on 
$\mathcal M_{\mathcal Q}$. 
\end{corollary}

\begin{proof} 
By Theorems \ref{thm:BDGGRS} and \ref{thm:Filip},  the Zariski closure of the monodromy on
$p (T \mathcal M_{\mathcal Q})= H^1_+(M,\C)$, hence on $H^1_+(M,\R)$, is the full symplectic group. Whenever the Zariski closure
of the monodromy on an invariant subbundle is the full symplectic group, the subbundle is strongly irreducible. In fact, as mentioned above, if the subbundle is not strongly irreducible, then, by \cite[Proposition B.4]{EFW18}, which in turn is derived from \cite[Theorem 7.7]{Fil16}, there exists a (local) labeling of the measurable subbundles in Definition~\ref{def:strongly_irr} such that each subbundle has a polynomial dependence on the base point. Since the symplectic group is Zariski-connnected, the invariance of a family of subbundles which depends polynomially on the base point would contradict the Zariski density of the monodromy group in the symplectic group. 
\end{proof}

\subsection*{The CLT for the KZ cocycle.} A (non-commutative) central limit theorem for the convergence to the top Lyapunov exponent of the KZ cocycle on subbundles of exterior powers of the cohomology bundle was proved in \cite{KZCLT} under several assumptions.

For the rest of this discussion fix a $d$-complex-dimensional, $\mathrm{SL}(2,\mathbb{R})$-invariant suborbifold $\mathcal{M}$ of a stratum of Abelian differentials with ergodic affine probability measure $\mu$. Let $\mathbf{H}$ be an $h$-dimensional, $\mathrm{SL}(2,\mathbb{R})$-invariant subbundle of the Kontsevich-Zorich cocycle over $\mathcal{M}$ with Lyapunov exponents $\lambda_1 \geq \lambda_2 \geq \cdots \geq \lambda_h$.  For every
$k \in \{1, \dots, h\}$, denote by $\smash{\textbf{H}^{(k)}}$ the $k$-th exterior power of the bundle $\textbf{H}$, let $\|\cdot\|$ be the natural extension of the Hodge norm to $\smash{\textbf{H}^{(k)}}$, and denote by $A^{(k)} := \{a_t \colon \smash{\textbf{H}^{(k)}} \to \smash{\textbf{H}^{(k)}}\}_{t \in \mathbb{R}}$ the lift of the KZ cocycle. Let $\widehat \mu$ be the unique measure on the projectivized bundle $\mathbb P \mathbf{H}^{(k)}$ which projects to $\mu$ and whose conditional
measures on fibers are equal to the Lebesgue (Haar) measures.

For every $(\omega,v)$ in $\mathbf{H}^{(k)}$ or $\mathbb{P}\mathbf{H}^{(k)}$ with $v \neq 0$ and every $t \in \mathbb{R}$ consider the quantities
\begin{gather*}
\sigma(\omega,v,t) := \log \frac{\|a_t v\|_{a_t \omega}}{\|v\|_{\omega}},\\
\sigma_k(\omega,t) := \log \sup_{v \in \mathbb{P}\mathbf{H}^{(k)}} \frac{\|a_t v\|_{a_t \omega}}{\|v\|_{\omega}}.
\end{gather*}

Recall that the top exponent of the KZ cocycle on $\mathbf{H}^{(k)}$ is $\Lambda_k := \sum_{i=1}^k \lambda_i$ and that, by the Oseledets theorem, for 
$\hat \mu$- almost every $(\omega, v) \in \mathbb{P}\mathbf{H}^{(k)}$, 
\begin{equation}
\label{eq:Oseledets}
\lim_{t\to \infty} \frac{\sigma(\omega,v,t)}{t}   = \sum_{i=1}^k \lambda_i\,.
\end{equation}

Denote by $\mathcal{N}(0,V)$ a Gaussian random variable of mean $0$ and variance $V > 0$

\begin{theorem} \cite[Theorem 2.1]{KZCLT}
\label{dclt}
Consider $\mathcal{M}$ an $\mathrm{SL}(2,\mathbb{R})$-invariant suborbifold of a stratum of Abelian differentials with ergodic affine probability measure $\mu$. Let $\textbf{H}$ be an $2h$-dimensional, strongly irreducible, symplectic, $\mathrm{SL}(2,\R)$-invariant subbundle of the Kontsevich-Zorich cocycle over $\mathcal{M}$ that is symplectically orthogonal to the tautological subbundle and with Lyapunov exponents $\lambda_1 \geq \cdots \geq \lambda_{2h}$. Fix $1 \leq k \leq h$ such that $\lambda_k > \lambda_{k+1}$, let $\smash{\Lambda_k := \sum_{i=1}^k \lambda_i}$, and consider the $k$-th exterior power $\mathbf{H}^{(k)}$ and its projectivization $\mathbb{P}\mathbf{H}^{(k)}$ endowed with the canonical measure $\widehat{\mu}$. Then there exists a real number $V_k
\geq 0$ such that the random variables
\[
S_t(\omega,v) := \frac {\sigma(\omega,v,t) - t \cdot \Lambda_k}{\sqrt{t}} \quad \text {on} \quad (\mathbb{P}\mathbf{H}^{(k)}, \widehat{\mu})
\]
converge in distribution to a Gaussian of mean $0$ and variance $V_k$ as $t \to \infty$, i.e., for every interval $(a,b) \subseteq \mathbb{R}$ such that $\mathbb{P}(\mathcal{N}(0,V_k) \in \{a,b\} = 0)$,
\[
\lim_{t \to \infty} \widehat{\mu}(\{ (\omega,v) \in \mathbf{H}^{(k)} \ | \ S_t(\omega,v) \in (a,b)\}) = \frac{1}{\sqrt{2\pi V_k}}\int_a^b e^{x^2/V_k} \thinspace dx.
\]
Moreover, if the KZ Lyapunov spectrum of $\mathbf{H}$ is simple, then $V_1> 0$.
\end{theorem}

\begin{remark}
The statement also holds in the event that $V_k = 0$, and in that case the resulting distribution is a delta distribution. The positivity of the variance holds for $2$-dimensional subbundles with strictly positive top Lyapunov exponent (for instance on the symplectic orthogonal of the tautological subbundle in genus $2$ for any $\text{\rm SL}(2,\R)$-invariant measure or on irreducible non-tautological subbundles for non-arithmetic Veech surfaces), as in this case the simplicity condition on the top  exponent is trivially satisfied.
\end{remark}

\begin{remark}
    The condition $\mathbb{P}(\mathcal{N}(0,V_k) \in \{a,b\}) = 0$ in Theorem \ref{theo:CLT1} is automatically satisfied if $V_k > 0$. If $V_k = 0$, this condition is equivalent to $a \neq 0$ and $b \neq 0$.
\end{remark}

We conclude this discussion mentioning some open questions. The first couple of questions are concerned with the generality of the conclusion of the above Theorem \ref{dclt}. 

\begin{question}
Does Theorem \ref{dclt} hold under the weaker assumption that the top exponent is strictly positive and simple, i.e., without the full simplicity assumption?
\end{question}

\begin{question}
   Under what conditions is the variance in Theorem \ref{dclt}  positive  for all exterior powers and not just for the first exterior power?  
\end{question}

The next question, perhaps more difficult, asks whether the CLT holds with respect to the Lebesgue measure on all
$\text{\rm SO}(2)$-orbits. Indeed, Chaika and Eskin
\cite{CE15} have proven a refinement of the Oseledets theorem for the KZ cocycle, showing
that the limit in formula \eqref{eq:Oseledets} holds for $r_\theta \omega:= e^{2\pi \imath \theta} \omega$ for Lebesgue almost every
$\theta \in [0, 2\pi)$ and every $\omega \in \mathcal M$. It is therefore natural to consider the following question.

\begin{question}
    Given an Abelian differential $\omega$, the union of fibers $\smash{\bigcup_{\theta \in \mathbb{R}/\mathbb{Z}} \mathbb{P}\mathbf{H}_{r_\theta \omega}}$ can be endowed with a natural measure $\widehat{\nu}$ which disintegrates as Lebesgue measure on the circle factor and Lebesgue measure on the fibers. Under what conditions do the random variables
    \[
    S_t(\theta,v) := \frac {\sigma(r_\theta\omega,v,t) - t \cdot \sum_{i=1}^k \lambda_i}{\sqrt{t}} \quad \text {on} \quad \left(\bigcup_{\theta \in \mathbb{R}/\mathbb{Z}} \mathbb{P}\mathbf{H}_{r_\theta \omega}, \widehat{\nu}\right)
    \]
    converge in distribution to a Gaussian random variable as $t \to \infty$?
\end{question}

\subsection*{Hyperbolicity, boundedness, and adaptedness.} We now check the conditions needed to upgrade limit theorems for the KZ cocycle to mixing limit theorems. Recall we have fixed a $d$-complex-dimensional, $\mathrm{SL}(2,\mathbb{R})$-invariant suborbifold $\mathcal{M}$ of a stratum of Abelian differentials with ergodic affine probability measure $\mu$, an $h$-dimensional, $\mathrm{SL}(2,\mathbb{R})$-invariant subbundle $\mathbf{H}$ of the Kontsevich-Zorich cocycle over $\mathcal{M}$, and an exterior power $\mathbf{H}^{(k)}$ of $\mathbf{H}$ for some $k \in \{1,\dots,h\}$.

First we introduce the relevant metric structure on $\mathcal{M}$. Recall that, given $\omega \in \mathcal{M}$ on a Riemann surface $M$ with zeroes and marked points $\Sigma \subseteq M$, the tangent space of $\omega$ at $\mathcal{M}$ can be identified with the cohomology group $H^1(M,\Sigma;\mathbb{C})$. For $v \in H^1(M,\Sigma;\mathbb{C})$ consider the norm defined as
\[
\|v\|_{\omega} := \sup_{\gamma \in \Gamma_{\omega}} \left| \frac{v(\gamma)}{\mathrm{hol}_\omega(\gamma)} \right|,
\]
where $\Gamma_{\omega} \subseteq H_1(M,\Sigma;\mathbb{Z})$ denotes the set of saddle connections of $\omega$ and $\mathrm{hol}_\omega(\gamma) \in \mathbf{C}$ denotes the holonomy of the saddle connection $\gamma$ with respect to $\omega$. By work of Avila, Gouëzel, and Yoccoz \cite{AGY06}, this definition indeed gives rise to a norm on $H^1(M,\Sigma;\mathbb{C})$ and the corresponding Finsler metric on $\mathcal{M}$ is complete. We refer to this metric as the AGY metric of $\mathcal{M}$ and denote it by $d_{\mathrm{AGY}}$.

Recall that $A:=\{a_t \colon \mathcal{M} \to \mathcal{M}\}_{t \in \mathbb{R}}$ denotes the Teichmüller geodesic flow. Denote by $U := h^-_1 \colon \mathcal{M} \to \mathcal{M}$ the time $1$ map of the stable horocycle flow. We begin by recalling the following direct consequence of the definition of the AGY-metric.

\begin{proposition} \cite[Lemma 5.2]{AG13}
\label{prop:bound}
Let $\mathcal{M}$ be an $\mathrm{SL}(2,\mathbb{R})$-invariant suborbifold of a stratum of Abelian differentials. Then, for every $\omega \in \mathcal{M}$ over a Riemann surface $M$ with zeros and marked points $\Sigma \subseteq M$, every $v \in H^1(M,\Sigma;i\mathbb{R})$, and every $t  \geq 0 $,
\[
\|v\|_{a_t\omega} \leq \|v\|_\omega.
\]
\end{proposition}

To apply our results from previous sections we need to upgrade Proposition \ref{prop:bound} to ensure contraction. The following result actually guarantees hyperbolicity.

\begin{proposition} \cite[Theorem 3.15]{ABEM12}
\label{prop:hyp}
 Let $\mathcal{M}$ be an $\mathrm{SL}(2,\mathbb{R})$-invariant suborbifold of a stratum of Abelian differentials with affine ergodic probability measure $\mu$. Then, there exists a constant $\kappa > 0$ such that for $\mu$-almost-every $\omega \in \mathcal{M}$ and every $t \geq 0$,
 \[
 d(a_t U \omega, a_t \omega) = O_\omega(e^{-\kappa t}).
 \]
\end{proposition}

\begin{proof}
    Let $\mathcal{K} \subseteq \mathcal{M}$ be a compact subset with $\mu(\mathcal{K}) \geq 2/3$. By Birkhoff's ergodic theorem, for $\mu$-almost-every $\omega \in \mathcal{M}$, the following asymptotic holds,
    \[
    \lim_{T \to \infty} |\{t \in [0,T] \colon a_t \omega \in \mathcal{K}\}| = 2/3 > 1/2.
    \]
    Let $\omega \in \mathcal{M}$ satisfy this condition. It follows from \cite[Proof of Theorem 3.15]{ABEM12}, Proposition \ref{prop:bound}, and the fact that Finsler metrics are comparable on compact sets, that there exists $\kappa = \kappa(\mathcal{K}) \geq 0$ such that for every $t \geq 0$,
    \[
    d_\mathrm{AGY}(a_t U \omega, a_t \omega) = O_\omega( e^{-\kappa t} d_\mathrm{AGY}(U\omega, \omega)) =  O_\omega( e^{-\kappa t}). \qedhere
    \]
\end{proof}

\begin{remark}
    More precise quantitative statements strengthening Proposition \ref{prop:hyp} hold. We refer the reader to \cite[Theorem 3.15]{ABEM12}, \cite[Proposition 2.7]{EMM19}, and \cite[Theorem 1.1]{A06} for stronger related results.
\end{remark}

Recall that $A^{(k)} := \{a_t \colon \smash{\mathbf{H}^{(k)}} \to \smash{\mathbf{H}^{(k)}} \}_{t \in \mathbb{R}}$ denotes the lift of the Kontsevich-Zorich cocycle to $\smash{\mathbf{H}^{(k)}}$. The main tool needed to prove boundedness of the cocycle $A^{(k)}$ is the following variational formula originally proved in the case $k=1$ by Forni \cite{For02} and extended to the case $k \geq 2$ by Forni, Matheus, Zorich \cite{FMZhodge}.

\begin{proposition}
    \label{prop:variational}
    \cite[Corollary 2.2]{FMZhodge} For every unit area Abelian differential $\omega \in \smash{\mathcal{H}_g}$ and every non-zero vector $v$ on the $k$-th exterior power of the Kontsevich-Zorich cocycle over $\mathcal{H}_g$, the following variational formula holds,
    \[
    \frac{d}{dt}\bigg|_{t = 0} \sigma(\omega,v,t) \leq k.
    \]
\end{proposition}

 Directly from Proposition \ref{prop:variational} we deduce the following corollary which, in particular, implies that $A^{(k)}$ is $\mathcal{V}$-sufficiently-bounded for every diverging sequence $\mathcal{V}$.

 \begin{corollary}
     \label{cor:bounded}
     For every Abelian differential $\omega \in \mathcal{H}_g$ and every non-zero vector $v$ on the $k$-th exterior power of the Kontsevich-Zorich cocycle over $\mathcal{H}_g$,
     \[
     \sigma(\omega,v,1) \leq k.
     \]
 \end{corollary}

Denote by $D$ the cocycle on $\mathbf{H}^{(k)}$ over $U = h_t^-$ given by parallel transport with respect to the Gauss-Manin connection along stable horocycle arcs. Recall that the projection of the $\mathrm{SL}(2,\mathbb{R})$-orbit of any marked, genus $g$ Abelian differential $\omega$ to the Teichmüller space of genus $g$ Riemann surfaces endowed with the Teichmüller metric is an embedded Poincaré disk; see for instance \cite{disk}. Recall also that the Hodge norm is $\mathrm{SO}(2)$-invariant. Directly from these facts and Proposition \ref{prop:variational} we deduce the cocycles $(A^{(k)}, D)$ are $(A,U,\mathcal{V})$-adapted for every diverging sequence $\mathcal{V}$.

\begin{corollary}
\label{cor:adapted}
For every Abelian differential $\omega \in \mathcal{H}_g$, every non-zero vector $v$ on the $k$-th exterior power of the Kontsevich-Zorich cocycle over $\mathcal{H}_g$, and every $t \geq 0$,
\[
|\sigma_k(x,v,t) - \sigma_k(Ux,Dv,t)| \leq k(1 + e^{-t}).
\]
\end{corollary}

\subsection*{The mixing laws of large numbers.} We now state a first set of distributional limit theorems for the Kontsevich-Zorich cocycle. These theorems can be interpreted as mixing laws of large numbers for exterior powers of subbundles of the KZ cocycle.

\begin{theorem}
\label{theo:LLN1}
    Consider $\mathcal{M}$ an $\mathrm{SL}(2,\mathbb{R})$-invariant suborbifold of a stratum of Abelian differentials with ergodic affine probability measure $\mu$. Let $\textbf{H}$ be an $h$-dimensional $\mathrm{SL}(2,\R)$-invariant subbundle of the Kontsevich-Zorich cocycle over $\mathcal{M}$ with Lyapunov exponents $\lambda_1 \geq \cdots \geq \lambda_h$. Fix $k \geq 1$ such that $\lambda_k > \lambda_{k+1}$, let $\smash{\Lambda_k := \sum_{i=1}^k \lambda_i}$, and consider the $k$-th exterior power $\mathbf{H}^{(k)}$ and its projectivization $\mathbb{P}\mathbf{H}^{(k)}$ endowed with the canonical measure $\widehat{\mu}$. Then, for the random variables
    \[
    S_t(\omega,v) := \frac {\sigma(\omega,v,t)}{t} \quad \text {on} \quad (\mathbb{P}\mathbf{H}^{(k)}, \widehat{\mu}),
    \]
    for every pair of Borel measurable subsets $B,E \subseteq \mathcal{M}$, and for every interval $(a,b)\subseteq \mathbb{R}$ such that $\Lambda_k \notin \{a,b\}$, the following holds,
    \begin{gather*}
    \lim_{t \to \infty} \widehat{\mu}(\{(\omega,v) \in \mathbb{P}\mathbf{H}^{(k)} \ | \ \omega \in B, \ S_t(\omega,v) \in (a,b), \ a_t\omega \in E\}) \\
    = \mu(B) \cdot \delta_{\Lambda_k}(a,b) \cdot \mu(E).
    \end{gather*}
\end{theorem}

\begin{proof}
    This is a direct consequence of the Oseledets ergodic theorem, Theorem \ref{theo:cocycle_MCLT_flows}, Lemma \ref{lemma:exp_flow}, Proposition \ref{prop:hyp}, and Corollaries \ref{cor:bounded} and \ref{cor:adapted}.
\end{proof}

We now state a mixing law of large numbers for the operator norm of the KZ cocycle.

\begin{theorem}
\label{theo:LLN2}
    Consider $\mathcal{M}$ an $\mathrm{SL}(2,\mathbb{R})$-invariant suborbifold of a stratum of Abelian differentials with ergodic affine probability measure $\mu$. Let $\textbf{H}$ be an $h$-dimensional $\mathrm{SL}(2,\R)$-invariant subbundle of the Kontsevich-Zorich cocycle over $\mathcal{M}$ with Lyapunov exponents $\lambda_1 \geq \cdots \geq \lambda_h$ such that $\lambda_k > \lambda_{k+1}$ and $\smash{\Lambda_k := \sum_{i=1}^k \lambda_i}$. Then, for the random variables
    \[
    S_t(\omega) := \frac {\sigma_k(\omega,t)}{t} \quad \text {on} \quad (\mathcal{M},\mu),
    \]
    for every pair of Borel measurable subsets $B,E \subseteq \mathcal{M}$, and for every interval $(a,b)\subseteq \mathbb{R}$ such that $\Lambda_k \notin \{a,b\}$, the following holds,
    \begin{gather*}
    \lim_{t \to \infty} \mu(\{ \omega \in \mathcal{M} \ | \ \omega \in B, \ S_t(\omega) \in (a,b), \ a_t\omega \in E\}) 
    = \mu(B) \cdot \delta_{\Lambda_k}(a,b) \cdot \mu(E).
    \end{gather*}
\end{theorem}

\begin{proof}
    This is a direct consequence of Theorems \ref{theo:cocycle_CCLT_norm2_flows} and \ref{theo:LLN1}, and Lemma \ref{lemma:exp2_flow}.
\end{proof}

We now state a mixing law of large numbers for generic sections of the KZ cocycle.

\begin{theorem}
\label{theo:LLN3}
    Consider $\mathcal{M}$ an $\mathrm{SL}(2,\mathbb{R})$-invariant suborbifold of a stratum of Abelian differentials with ergodic affine probability measure $\mu$. Let $\textbf{H}$ be an $h$-dimensional $\mathrm{SL}(2,\R)$-invariant subbundle of the Kontsevich-Zorich cocycle over $\mathcal{M}$ with Lyapunov exponents $\lambda_1 \geq \cdots \geq \lambda_h$. Fix $k \geq 1$ such that $\lambda_k > \lambda_{k+1}$, let $\smash{\Lambda_k := \sum_{i=1}^k \lambda_i}$, and consider the $k$-th exterior power $\mathbf{H}^{(k)}$. Let $s \colon \mathcal{M} \to \mathbf{H}^{(k)}$ be a section such that $s(\omega) \in \mathbf{H}^{(k)}$ is future-Oseledets-generic for $\mu$-almost every $\omega \in \mathcal{M}$. Then, for the random variables
    \[
    S_t(\omega) := \frac {\sigma_k(\omega,s(\omega),t)}{t} \quad \text {on} \quad (\mathcal{M},\mu),
    \]
    for every pair of Borel measurable subsets $B,E \subseteq \mathcal{M}$, and for every interval $(a,b)\subseteq \mathbb{R}$ such that $\Lambda_k \notin \{a,b\}$, the following holds,
    \begin{gather*}
    \lim_{t \to \infty} \mu(\{ \omega \in \mathcal{M} \ | \ \omega \in B, \ S_t(\omega) \in (a,b), \ a_t\omega \in E\}) 
    = \mu(B) \cdot \delta_{\Lambda_k}(a,b) \cdot \mu(E).
    \end{gather*}
\end{theorem}

\begin{proof}
    This is a direct consequence of Theorems \ref{theo:cocycle_CCLT_sec2_flows} and \ref{theo:LLN1}, and Lemma\ref{lemma:exp3_flow}.
\end{proof}

\subsection*{The mixing central limit theorems.} We now state a second set of distributional limit theorems for the Kontsevich-Zorich cocycle. These theorems correspond to mixing central limit theorems for exterior powers of subbundles of the KZ cocycle.

\begin{theorem}
    \label{theo:CLT1}
    Consider $\mathcal{M}$ an $\mathrm{SL}(2,\mathbb{R})$-invariant suborbifold of a stratum of Abelian differentials with ergodic affine probability measure $\mu$. Let $\textbf{H}$ be a $2h$-dimensional, strongly irreducible, symplectic, $\mathrm{SL}(2,\R)$-invariant subbundle of the Kontsevich-Zorich cocycle over $\mathcal{M}$ that is symplectically orthogonal to the tautological subbundle and with Lyapunov exponents $\lambda_1 \geq \cdots \geq \lambda_{2h}$. Fix $1 \leq k \leq h$ such that $\lambda_k > \lambda_{k+1}$, let $\smash{\Lambda_k := \sum_{i=1}^k \lambda_i}$, and consider the $k$-th exterior power $\mathbf{H}^{(k)}$ and its projectivization $\mathbb{P}\mathbf{H}^{(k)}$ endowed with the canonical measure $\widehat{\mu}$. Then, there exists $V_k \geq 0$ such that for the random variables
    \[
    S_t(\omega,v) := \frac {\sigma(\omega,v,t) - t \cdot \Lambda_k}{\sqrt{t}} \quad \text {on} \quad (\mathbb{P}\mathbf{H}^{(k)}, \widehat{\mu}),
    \]
    for every pair of Borel measurable subsets $B,E \subseteq \mathcal{M}$, and for every interval $(a,b)\subseteq \mathbb{R}$ such that such that $\mathbb{P}(\mathcal{N}(0,V_k) \in \{a,b\}) = 0$, the following holds,
    \begin{gather*}
    \lim_{t \to \infty} \widehat{\mu}(\{(\omega,v) \in \mathbb{P}\mathbf{H}^{(k)} \ | \ \omega \in B, \ S_t(\omega,v) \in (a,b), \ a_t\omega \in E\}) \\
    = \mu(B) \cdot \left(\frac{1}{\sqrt{2\pi V_k}} \int_a^b e^{-x^2/V_k} \thinspace dx \right) \cdot \mu(E).
    \end{gather*}
    Moreover, if the Lyapunov spectrum of $\mathbf{H}$ is simple, then $V_1 > 0$.
\end{theorem}

\begin{proof}
    This is a direct consequence of Theorems \ref{dclt} and \ref{theo:cocycle_MCLT_flows}, Lemma \ref{lemma:exp_flow}, Proposition \ref{prop:hyp}, and Corollaries \ref{cor:bounded} and \ref{cor:adapted}.
\end{proof}

We now state a mixing central limit theorem for the Kontsevich-Zorich cocycle with respect to the operator norm.

\begin{theorem}
    \label{theo:CLT2} 
    Consider $\mathcal{M}$ an $\mathrm{SL}(2,\mathbb{R})$-invariant suborbifold of a stratum of Abelian differentials with ergodic affine probability measure $\mu$. Let $\textbf{H}$ be a $2h$-dimensional, strongly irreducible, symplectic, $\mathrm{SL}(2,\R)$-invariant subbundle of the Kontsevich-Zorich cocycle over $\mathcal{M}$ that is symplectically orthogonal to the tautological subbundle and with Lyapunov exponents $\lambda_1 \geq \cdots \geq \lambda_{2h}$. Fix $1 \leq k \leq h$ such that $\lambda_k > \lambda_{k+1}$ and let $\smash{\Lambda_k := \sum_{i=1}^k \lambda_i}$. Then, there exists $V_k \geq 0$ such that for the random variables
    \[
    S_t(\omega) := \frac {\sigma_k(\omega,t) - t \cdot \Lambda_k}{\sqrt{t}} \quad \text {on} \quad (\mathcal{M},\mu),
    \]
    for every pair of Borel measurable subsets $B,E \subseteq \mathcal{M}$, and for every interval $(a,b)\subseteq \mathbb{R}$ such that $\mathbb{P}(\mathcal{N}(0,V_k) \in \{a,b\}) = 0$, the following holds,
    \begin{gather*}
    \lim_{t \to \infty} \mu(\{\omega\in \mathcal{M} \ | \ \omega \in B, \ S_t(\omega) \in (a,b), \ a_t\omega \in E\}) \\
    = \mu(B) \cdot \left(\frac{1}{\sqrt{2\pi V_k}} \int_a^b e^{-x^2/V_k} \thinspace dx \right) \cdot \mu(E).
    \end{gather*}
    Moreover, if the Lyapunov spectrum of $\mathbf{H}$ is simple, then $V_1 > 0$.
\end{theorem}

\begin{proof}
    This is a direct consequence of Theorems \ref{theo:cocycle_CCLT_norm2_flows} and \ref{theo:CLT1}, and Lemma \ref{lemma:exp2_flow}.
\end{proof}

Finally, we state a central limit theorem for generic sections of the KZ cocycle.

\begin{theorem}
    \label{theo:CLT3}
    Consider $\mathcal{M}$ an $\mathrm{SL}(2,\mathbb{R})$-invariant suborbifold of a stratum of Abelian differentials with ergodic affine probability measure $\mu$. Let $\textbf{H}$ be an $2h$-dimensional, strongly irreducible, symplectic, $\mathrm{SL}(2,\R)$-invariant subbundle of the KZ cocycle over $\mathcal{M}$ that is symplectically orthogonal to the tautological subbundle and with Lyapunov exponents $\lambda_1 \geq \cdots \geq \lambda_{2h}$. Fix $1 \leq k \leq h$ such that $\lambda_k > \lambda_{k+1}$ and let $\smash{\Lambda_k := \sum_{i=1}^k \lambda_i}$. Then, there exists $V_k \geq 0$ such that for every measurable section $s \colon \mathcal{M} \to \mathbf{H}^{(k)}$ with $s(\omega) \in \mathbf{H}^{(k)}$ future-Oseledets-generic for $\mu$-almost every $\omega \in \mathcal{M}$, for the random variables
    \[
    S_t(\omega) := \frac {\sigma_k(\omega, s(\omega),t) - t \cdot \Lambda_k}{\sqrt{t}} \quad \text {on} \quad (\mathcal{M},\mu),
    \]
    for every pair of Borel measurable subsets $B,E \subseteq \mathcal{M}$, and for every interval $(a,b)\subseteq \mathbb{R}$ such that $\mathbb{P}(\mathcal{N}(0,V_k) \in \{a,b\}) = 0$, the following holds,
    \begin{gather*}
    \lim_{t \to \infty} \mu(\{\omega\in \mathcal{M} \ | \ \omega \in B, \ S_t(\omega) \in (a,b), \ a_t\omega \in E\}) \\
    = \mu(B) \cdot \left(\frac{1}{\sqrt{2\pi V_k}} \int_a^b e^{-x^2/V_k} \thinspace dx \right) \cdot \mu(E).
    \end{gather*}
    Moreover, if the Lyapunov spectrum of $\mathbf{H}$ is simple, then $V_1 > 0$.
\end{theorem}

\begin{proof}
    This is a direct consequence of Theorems \ref{theo:cocycle_CCLT_sec2_flows} and \ref{theo:CLT1}, and Lemma \ref{lemma:exp3_flow}.
\end{proof}

\subsection*{A criterion for genericity of sections.}

We recall below results on Oseledets genericity for the Kontsevich--Zorich cocycles on strongly irreducible subbundles of the Hodge bundle. We begin with the following result concerning the deviations of Lyapunov exponents.


\begin{theorem} [Theorem 1.5, \cite{AlSaqbanetal}]

Consider $\mathcal{M}$ an $\mathrm{SL}(2,\mathbb{R})$-invariant suborbifold of a stratum of Abelian differentials with ergodic affine probability measure $\mu$. Let $\textbf{H}$ be an $h$-dimensional, $\mathrm{SL}(2,\R)$-invariant, continuous subbundle of the Kontsevich-Zorich cocycle over $\mathcal{M}$ with Lyapunov exponents $\lambda_1 \geq \cdots \geq \lambda_h$. Fix $k \in \{1,\dots,h\}$, let $\smash{\Lambda_k := \sum_{i=1}^k \lambda_i}$, and consider the $k$-th exterior power $\mathbf{H}^{(k)}$. Assume the Kontsevich-Zorich cocycle on $\mathbf{H}^{(k)}$ is strongly irreducible with respect to $\mu$. Then, for any $\epsilon > 0$, there exist affine invariant submanifolds $\mathcal N_1, \dots, \mathcal N_k$ properly contained in $\mathcal M$, and $0<\delta <1$, such that for all $\omega \in {\mathcal M} \setminus \bigcup_{i=1}^k \mathcal N_i$, the set
\[
\left\lbrace \theta \in [0, 2\pi) \ \vert \  \limsup_{t\to \infty} \frac{\sigma_k(r_\theta\omega, t) }{t} \geq \Lambda_k + \epsilon \right\rbrace 
\]
has Hausdorff dimension at most $\delta$.
\end{theorem}

The theorem above is a refinement of a fundamental result of Chaika and Eskin who
in \cite[Theorem 1.5]{CE15} proved that, for 
all $\omega \in {\mathcal M} \setminus \bigcup_{i=1}^k \mathcal N_i$ and for almost all $\theta\in [0, 2\pi)$,
$$
\limsup_{t\to \infty} \frac{\sigma_k(r_\theta\omega, t) }{t} = \Lambda_k\,.
$$

A similar result can be stated for horocycle arcs.

\begin{theorem} [Theorem 1.5, \cite{AlSaqbanetal}] Consider $\mathcal{M}$ an $\mathrm{SL}(2,\mathbb{R})$-invariant suborbifold of a stratum of Abelian differentials with ergodic affine probability measure $\mu$. Let $\textbf{H}$ be an $h$-dimensional, $\mathrm{SL}(2,\R)$-invariant, continuous subbundle of the Kontsevich-Zorich cocycle over $\mathcal{M}$ with Lyapunov exponents $\lambda_1 \geq \cdots \geq \lambda_h$. Fix $k \in \{1,\dots,h\}$, let $\smash{\Lambda_k := \sum_{i=1}^k \lambda_i}$, and consider the $k$-th exterior power $\mathbf{H}^{(k)}$. Assume the Kontsevich-Zorich cocycle on $\mathbf{H}^{(k)}$ is strongly irreducible with respect to $\mu$. Then, for any $\epsilon > 0$, there exist affine invariant submanifolds $\mathcal N_1, \dots, \mathcal N_k$ properly contained in $\mathcal M$, and $0<\delta <1$, such that for all $\omega \in {\mathcal M} \setminus \bigcup_{i=1}^k \mathcal N_i$, the set
\[
\left\lbrace s \in \mathbb{R} \ \vert \  \limsup_{t\to \infty} \frac{\sigma_k(h_s^+ \omega, t) }{t} \geq \Lambda_k + \epsilon \right\rbrace 
\]
has Hausdorff dimension at most $\delta$.
\end{theorem}

In the other direction the following large deviation result holds.

\begin{lemma}  
Consider $\mathcal{M}$ an $\mathrm{SL}(2,\mathbb{R})$-invariant suborbifold of a stratum of Abelian differentials with ergodic affine probability measure $\mu$. Let $\textbf{H}$ be an $h$-dimensional, $\mathrm{SL}(2,\R)$-invariant, continuous subbundle of the Kontsevich-Zorich cocycle over $\mathcal{M}$ with Lyapunov exponents $\lambda_1 \geq \cdots \geq \lambda_h$. Fix $k \in \{1,\dots,h\}$, let $\smash{\Lambda_k := \sum_{i=1}^k \lambda_i}$, and consider the $k$-th exterior power $\mathbf{H}^{(k)}$. Assume the KZ cocycle on $\mathbf{H}^{(k)}$ is strongly irreducible with respect to $\mu$. Let $s : \mathcal M \to \textbf{H}^{(k)}$ be a measurable, nowhere vanishing section.  Then,
\begin{enumerate}
\item for any $k\geq 1$, if the section $s$ is $\mathrm{SO}(2)$-invariant or $H^+$-invariant for $\mu$-almost every $\omega \in \mathcal{M}$, then, 
for every $\epsilon >0$, there exists $C>1$ such that for every $t>0$, 
\[
\mu \left( \left\lbrace \omega \in \mathcal M \colon
\sigma(\omega, s(\omega), t)\leq  e^{(\Lambda_k -\epsilon)t } \right\rbrace \right)  \leq  C e^{-t/C}\,;
\]
\item for $k= 1$, the above large deviation estimate holds under the weaker hypothesis that the section $s$ is $\mathrm{SO}(2)$-Lipschitz-continuous or $H^+$-Lipschitz-continuous for $\mu$-almost-every $\omega \in \mathcal M$. 
\end{enumerate}
\end{lemma}

\begin{proof} The first item is proved for parallel sections above horocycle arcs with respect to the Lebesgue measure in \cite[Remark 5.4]{CF}. Hence, the statement for sections which restrict to parallel sections along horocycle arcs follows immediately by Fubini's theorem.

The statement for constant sections along circle orbits is equivalent since we have 
$$
 r_\theta= g_{\thinspace \log \vert \cos\theta \vert} \circ h^-_{-\sin \theta \cos\theta } \circ h^+_{\tan\theta}  \,.
$$
Hence, for any $\omega \in \mathcal{M}$, the forward Teichm\"uller orbit of $r_\theta \omega$ is exponentially asymptotic to that of the Abelian differential $g_{\thinspace \log \vert \cos\theta \vert} \circ  h^+_{\tan\theta} \thinspace \omega$, which implies that $r_\theta\omega$ is Oseledets regular if and only if 
$h^+_{\tan\theta} \thinspace \omega$ is. In addition, the parallel transport on the Hodge bundle of a constant section on the circle $\mathrm{SO}(2) \omega$ to the horocycle orbit $H^+ \omega$
is by definition a parallel section
above $H^+ \omega$.
Since the arctangent function is absolutely continuous, 
the large deviation statement for
constant sections above circles follows from the analogous statement for parallel sections above horocycles. 
 The statement in the first item 
 for sections which are constant above almost all circles then follows by Fubini's theorem. 

The second item can be proved for parallel sections above horocycle orbits as explained in~\cite[Remark 5.4]{CF}.
The argument can be reduced to invariant  subbundles symplectically orthogonal to the tautological subbundle. In the latter case, the argument is based on \cite[Lemma 5.3]{CF} and on a `freezing' mechanism (suggested by Jon Chaika
and exploited in \cite{CF} and \cite{KZCLT}).  The freezing mechanism is based on the observation that since the growth of the length of unstable horocycle arcs dominates the growth of vectors under the action of the cocycle, the image of a Lipschitz section under the Teichm\"uller geodesic flow  is eventually projectively Lipschitz
(in the sense of \cite[Definition 4.3]{CF}). Then \cite[Lemma 5.3]{CF} applies to projectively Lipschitz sections and implies the desired large deviations estimate for paralellel sections along horocycle
orbits. A similar statement can then be derived for constant sections along circle orbits 
as outlined above. Finally, the stated measure estimate holds for Lipschitz sections by Fubini's theorem. 
\end{proof}     

In summary, we have the following genericity criterion.

\begin{theorem}
Consider $\mathcal{M}$ an $\mathrm{SL}(2,\mathbb{R})$-invariant suborbifold of a stratum of Abelian differentials with ergodic affine probability measure $\mu$. Let $\textbf{H}$ be an $h$-dimensional, $\mathrm{SL}(2,\R)$-invariant, continuous subbundle of the Kontsevich-Zorich cocycle over $\mathcal{M}$ with Lyapunov exponents $\lambda_1 \geq \cdots \geq \lambda_h$. Fix $k \in \{1,\dots,h\}$, let $\smash{\Lambda_k := \sum_{i=1}^k \lambda_i}$, and consider the $k$-th exterior power $\mathbf{H}^{(k)}$. Assume the KZ cocycle on $\mathbf{H}^{(k)}$ is strongly irreducible with respect to $\mu$. Let $s : \mathcal M \to \textbf{H}^{(k)}$ be a measurable, nowhere vanishing section.  Then,
\begin{enumerate}
\item for any $k\geq 1$, if the section $s$ is $\mathrm{SO}(2)$-invariant or $H^+$-invariant for $\mu$-almost every $\omega \in \mathcal{M}$, then the section $s$ is generic, i.e., 
$$
\lim_{t\to \infty} \frac{\sigma(\omega, s(\omega),t)}{t} = \Lambda_k\,  \quad \text{ for $\mu$-almost every }  \omega \in \mathcal M \,.
$$
\item for $k= 1$, the above conclusion holds  under the weaker hypothesis that the section $s$ is $\mathrm{SO}(2)$-Lipschitz or $H^+$-Lipschitz for $\mu$-almost-every $\omega \in \mathcal M$.  
\end{enumerate}
\end{theorem}

\bibliographystyle{amsalpha}


\bibliography{bibliography}

$ $

\end{document}